\documentclass{jcmlatex}
\def\JCMvol{xx}
\def\JCMno{x}
\def\JCMyear{200x}
\def\JCMreceived{xxx}
\def\JCMrevised{xxx}
\def\JCMaccepted{xxx}

\usepackage{graphicx}
\usepackage{subfigure}
\usepackage{comment}
\usepackage{amsmath}
\usepackage{enumerate}
\usepackage{float}
\usepackage{bm}
\usepackage{dsfont}
\usepackage{amssymb}
\usepackage{mathrsfs}
\usepackage{amsfonts}
\usepackage{galois}
\usepackage[labelsep=space]{caption}
\usepackage{color}
\usepackage[titletoc]{appendix}
\newcommand{\norm}[1]{\left\Vert#1\right\Vert}
\newcommand{\abs}[1]{\left\vert#1\right\vert}

\newcommand{\qed}{\hfill $\square$}

\begin{document}

\markboth{J.~JIANG, Y.~SHI, X.~WANG AND X.~CHEN}{Two-stage stochastic variational inequalities for Nash equilibrium}

\title{REGULARIZED TWO-STAGE STOCHASTIC VARIATIONAL INEQUALITIES FOR COURNOT-NASH EQUILIBRIUM UNDER UNCERTAINTY}

\author{Jie Jiang, Yun Shi, Xiaozhou Wang and Xiaojun Chen
\thanks{Department of Applied Mathematics, The Hong Kong Polytechnic University, Hong Kong, China \\ Email: jie.jiang@connect.polyu.hk, yun.shi@connect.polyu.hk, xzhou.wang@connect.polyu.hk, xiaojun.chen@polyu.edu.hk}
}

\maketitle

\begin{abstract}
A convex two-stage non-cooperative multi-agent game under uncertainty is formulated as a two-stage stochastic variational inequality (SVI).
Under standard assumptions, we provide sufficient conditions for the existence of solutions of the two-stage SVI and propose a regularized sample average approximation method for solving  it.
We prove the convergence of the method as the regularization parameter tends to zero and the sample size tends to infinity.
Moreover, our approach is applied to a two-stage stochastic production and supply planning problem with homogeneous commodity in an oligopolistic market.
Numerical results based on historical data in crude oil market are presented to demonstrate the effectiveness of the two-stage SVI in describing the market share of oil producing agents.
\end{abstract}



\section{Introduction}
\label{sec:intro}
We consider a two-stage non-cooperative multi-agent game under uncertainty.
The model is used to describe a homogenous commodity production and supply planning problem in an oligopolistic market.
In particular, we focus on a $J$-player stochastic Nash equilibrium problem (SNEP) of Cournot competition, whose solution concept is characterized by stochastic Cournot-Nash (CN) equilibria.
Conventional non-cooperative game theory has a long history being an effective model to describe market behaviour, yet mostly under deterministic settings.
In order to explore characteristics of real markets, one cannot neglect the presence of uncertainty.

Researchers have studied various real markets through SNEPs in recent years.
Jofr{\'e}, Rockafellar and Wets \cite{JRW2007variational} investigated various economic equilibria using SVIs.
A scenario-based multi-stage oligopolistic market equilibrium problem under uncertainty was discussed in \cite{GRS2007dynamic}.
A two-settlement oligopolistic equilibrium with uncertainty in the future market was presented in \cite{YAO2008modeling}.
For practical applications, electricity markets with hydro-electric distribution have been studied by Philpott, Ferris and Wets \cite{PFW2016equilibrium} in which the levels of water reserves were modelled under uncertainty.
The contemporary treatment of classical equilibrium problems are investigated through finite-dimensional variational inequalities (VIs) and complementarity problems (CPs) with a wide range of applications under the assumption of deterministic and single-stage decision, see \cite{HP1990finite,FP2007finite} and references therein.
Yet, many practical applications are formulated under uncertainty \cite{HP2007nash,CWZ2012stochastic,PSL2015constructive,XPOD2015complementarity,CSX2017discrete,PSS2017two}, and the corresponding multi-stage SVIs and CPs were studied extensively over the last two decades.
In \cite{CF2005expected}, Chen and Fukushima considered the stochastic linear CP (LCP) by expected residual minimization (ERM) procedure.
The quasi-Monte Carlo method was adopted to generate scenarios of observation and thus to obtain the discrete approximation problem.
Chen, Wets and Zhang \cite{CWZ2012stochastic} investigated SVI problems by the ERM procedure, and the sample average approximation (SAA) method was employed to approximate the expected smoothing residual function.
More recently, as an extension from single-stage to multi-stage decision processes, Rockafellar and Wets \cite{RW2017stochastic} put forward multi-stage SVI problems, which laid a solid foundation for further research.
In \cite{CPW2017two}, Chen, Pang and Wets introduced the ERM procedure for two-stage SVI problems and the Douglas-Rachford splitting method was used to present numerical results.
Chen, Sun and Xu \cite{CSX2017discrete} considered  a two-stage stochastic LCP.
Structural properties of the problem were studied under the assumption of strong monotonicity, and a discrete scheme was conducted by partition of the support set and the corresponding convergence assertion was established.
More generally, Chen, Shapiro and Sun \cite{CSS2018convergence} investigated  the SAA of two-stage stochastic nonlinear generalized equations, which included two-stage nonlinear SVI problems as special cases.
Exponential rate of convergence was derived by using the technique of perturbed partial linearization.
From the perspective of the numerical calculation, the equilibrium problems are usually rewritten as a minimization problem, mostly nonsmooth.
For this class of problems, the smoothing techniques (see \cite{C2012smoothing}) can be employed so that differentiable methods, e.g., Newton's method, become applicable in solving the smoothing problem, see for instance \cite{CY1999homotopy,CY2000smoothing,CF2004smoothing,CX2013newton}.
For numerical implementation, Rockafellar and Sun extended the well-known progressive hedging method (PHM) for multi-stage stochastic programming problems to multi-stage  SVIs in \cite{RS2018solving}.

The decision vectors of production and supply plan problems, so-called \emph{strategies}, are distinguished into two categories: (i) those of ``here-and-now'' type, which do not depend on outcomes of random events in the future, and (ii) those treated as responses.
The goal of this paper is to establish a model that describes the market mechanism for Cournot competition under uncertainties.
The solution concepts of the model assemble the real strategies adopted by participants of the market.
We will provide methodology to solve the proposed model and demonstrating its numerical implementation.

We summarize the main contributions of this paper as follows.
\begin{itemize}
\item A two-stage stochastic Nash equilibrium problem is proposed to model production and supply competition of a homogenous product under uncertainty in an oligopolistic market.
\item The model is  recast as two-stage SVIs whose solutions characterize a  CN equilibrium.
\item A regularized sample average approximation method is proposed to solve the two-stage SVIs with convergence properties under mild assumptions.
\item The model is tested numerically for its effectiveness.  Moreover, it is used to describe the market share observation in the world market of crude oil.  We show that our model is not only able to reproduce historical in-sample market share but also capable of making out-of-sample predictions based on real data sets.
\end{itemize}
The remaining of the paper is organized as follows. In section 2, the two-stage SNEPs are developed and recast into two-stage SVIs. Section 3 contains structural analysis of two-stage SVIs and the corresponding regularized problems. The convergence theorems of the proposed regularized SAA method are presented in section 4.  Numerical results based on randomly generated data and simulation on real market data of crude oil are presented and analyzed in section 6.  We conclude our study in the final section.

\section{Two-stage stochastic Cournot-Nash game}
All stochastic models involve inherently ``ordered'' components over decision horizons.
In particular, the decisions in a strategy may respond to the information that is available only at the present stage, which is known as \emph{nonanticipativity}.
In this paper, we consider a \emph{two-stage stochastic CN equilibrium problem}, which extends the classical deterministic CN equilibrium problem in \cite{MSS1982mathematical}.
Let $\xi:\Omega\to\mathbb{R}^d$ be a random variable with support set $\Xi\subseteq\mathbb{R}^d$, and let $(\Xi,\mathcal{F},P)$ denote the induced probability space.
A strategy pair of agent $j\in\mathcal{J}:=\{1,...,J\}$ is denoted as
\begin{equation}\label{strategy}
(x_j\in\mathbb{R}, y_j:\Xi\to\mathbb{R},\mathcal{F}\text{-measurable}),
\end{equation}
where $x_j$ is a first-stage decision vector and $y_j\in\mathcal{Y}$ represents a second-stage \emph{response function} with $\mathcal{Y}$ being the space of $\mathcal{F}$-measurable functions defined on $\Xi$.
Let $\mathfrak{L}_n$ be the Lebesgue space of $\mathbb{R}^n$-valued functions with $\mathfrak{L}^\infty_n$ denotes the class of measurable essentially bounded functions.
Following a similar treatment as in \cite{CPW2017two}, we further require the second-stage response function of random variable to be essentially bounded, i.e., $y_j\in\mathfrak{L}^\infty_1$.
Collectively, the vector of strategy pairs of all agents can be written as
\begin{equation}\label{strategyvector}
(x\in\mathbb{R}^J, y:\Xi\to\mathbb{R}^J,\mathcal{F}\text{-measurable}).
\end{equation}
A \emph{strategy pair} $\big(x_j^*,y_j^*\big)\in\mathbb{R}\times\mathfrak{L}^\infty_1$ is said to be an equilibrium of our stochastic model if it solves the following problem for all agents $j\in\mathcal{J}$.
\begin{equation}\label{twostage}\begin{aligned}
&\text{maximize}_{(x_j, y_j(\cdot))} \hspace{.25cm}&&W^1_j(x_j,x^*_{-j})+\mathbb{E}\big[W^2_j\big(\xi, y_j(\xi),y^*_{-j}(\xi)\big)\big], &&\text{(objective function)}\\
&\text{subject to}\hspace{1.5cm}&&x_j \in X_j,~ &&\text{(first-stage constraints)}\\
&\hspace{3cm}&&y_j(\xi)\in_{a.s.} Y_j,~~g_j\big(\xi,x_j,y_j(\xi)\big)\leq_{a.s.}0 &&\text{(second-stage constraints)}\\
\end{aligned}\end{equation}
where
\[\begin{aligned}
&x^*_{-j} &&= &&(x^*_1,\ldots,x^*_{j-1},x^*_{j+1},\ldots,x^*_J),\\
&y^*_{-j}(\xi) &&= &&(y^*_1(\xi),\ldots,y^*_{j-1}(\xi),y^*_{j+1}(\xi),\ldots,y^*_J(\xi)),
\end{aligned}\]
and $y_j(\xi)$ denotes the value of response $y_j$ to realization $\xi$\footnote{For ease the exposition, the same notation $\xi$ is used for both a random variable and its realization in $\mathbb{R}^d$ without causing confusion on the context.}, with
\begin{itemize}
\item $W^1_j:\mathbb{R}\times\mathbb{R}^{J-1}\to\mathbb{R}$ a first-stage wealth function of agent $j$, concave and continuously differentiable with respect to $x_j$.
\item $W^2_j:\Xi\times\mathbb{R}\times\mathbb{R}^{J-1}\to\mathbb{R}$ a second-stage wealth function of agent $j$, concave, well-defined and finite.
\item  $X_j, Y_j$  nonempty, closed and convex subsets of $\mathbb{R}$ and the second-stage constraints hold almost surely ($a.s.$).
\item $g_j:\Xi\times\mathbb{R}\times\mathbb{R}\to\mathbb{R}$ a continuously differentiable function with respect to $\big(x_j, y_j(\xi)\big)$ for almost every (a.e.) $\xi\in \Xi$ and $\mathcal{F}$-measurable.
\end{itemize}
In this paper, the model (\ref{twostage}) is formulated under the assumption that the uncertainty can be described by a random variable $\xi$ with known distribution.
From the perspective of the entire system, the market\footnote{A ``social planner'' is commonly termed in literature of economics, and can be interpreted as one individual who oversees the system and assigns strategy to each agents.} ``chooses'' $x\in\mathbb{R}^J$ before a realization $\xi\in\Xi$ is revealed and later ``selects'' $y(\xi)\in\mathbb{R}^J$ with known realization.

\subsection{Two-stage stochastic commodity production and supply planning}
The application of commodity production and supply in an oligopolistic market serves as a motivation as well as the practical problem of interest.
Presented as a stochastic game, the strategy of each agent in supply-side of the market can be described as the solution of a stochastic optimization problem (\ref{twostage}).
The decision process follows that agent $j\in\mathcal{J}$ decides an optimal production quantity $x_j$ of the commodity at the production stage.
At the second-stage, each agent decides a supply quantity $y_j(\xi)$ after $\xi$ is observed, and a total quantity $T(y(\xi)):=\sum_{j=1}^J{y_j(\xi)}$ is supplied to the market.
Our focus on oligopolistic markets requires that the price is dominantly affected by the total supplied quantity in the market $T(y(\xi))$.
Therefore, all the trading occurs at the price $p:\Xi\times\mathbb{R}\to\mathbb{R}_+$, determined by a stochastic inverse demand curve $p(\xi, T(y(\xi))$.
In practice, production and supply quantities are subject to physical restrictions, e.g., capability of production plant, logistic restriction, etc., i.e., $x_j\in X_j$ and ${y_j(\xi)} \in_{a.s.} Y_j$.
More specifically, we have non-negative requirements for both production and supply, $X_j=\mathbb{R}_+$ and $Y_j=\mathbb{R}_+$.
The relations between stage-wise decision variables $x_j$ and $y_j(\xi)$ are captured by constraints $g_j\big(\xi,x_j,{y_j(\xi)}\big)\leq_{a.s.}0$ in (\ref{twostage}).
Essentially, every agent needs to formulate and solve a \emph{two-stage stochastic programming problem with recourse} in the sense of achieving equilibria of a $J$-agents non-cooperative game of the market.
We further require that agent $j$'s supply to the market cannot exceed his/her production quantity, i.e., ${y_j(\xi)}-x_j\leq_{a.s.}0$.
This can be interpreted as the fact that agents may have no stock to start with, or they need to preserve certain reserved quantities prior to each decision process.

The problem can then be viewed from a slight different perspective than that of problem (\ref{twostage}).
As seen from the first-stage, agent $j\in\mathcal{J}$ wants to find a production quantity $x_j\geq0$ to
\begin{equation}\label{GM1}
\text{maximize}~W^1_j(x_j,x^*_{-j})+\mathbb{E}\big[\Phi_j(x_j,x^*_{-j},\xi)\big],
\end{equation}
where,
\begin{equation}\label{GM2}
\Phi_j(x_j,x^*_{-j},\xi) = \sup_{{y_j(\xi)}\geq0}\{ W^2_j(\xi, {y_j(\xi)},{y^*_{-j}(\xi)})~|~x_j\geq {y_j(\xi)},\text{~for a.e.~}\xi\in\Xi\}.
\end{equation}
Objective function in (\ref{GM1}) is regarded as the expected profit of agent $j$'s and problems (\ref{GM1})-(\ref{GM2}) are termed \emph{intrinsic first-stage problem} following that of a related treatment in \cite{RW1976stochastic}.
In particular, the analysis of intrinsic first-stage problem and the stochastic programming problem with recourse in convex case were carried out in a series of studies by Rockafellar and Wets \cite{RW1975stochastic,RW1976stochastic,RW1976stochastic1,RW1976stochastic2} and more recently in \cite{RW2017stochastic}.
The key feature of intrinsic first-stage problem, as well as formulation (\ref{twostage}), is the requirement on precise orders of decision execution, commonly known as the constraints of \emph{nonanticipativity}, which implies the decision variable in each stage should only depend on the information (data) available up to that stage.
In (\ref{GM1})-(\ref{GM2}), the second-stage decisions are explicitly determined after the first-stage decision, provided for each $x_j$ the second-stage problem is well-defined \cite{SDR2009lectures}.
However, the study of optimality condition of (\ref{GM1})-(\ref{GM2}), in the case of a general probability space $(\Xi,\mathcal{F},P)$\footnote{In cases of finitely supported distribution, the equivalence between intrinsic first-stage problem and the original recourse problem can be established, and the optimality condition of the recourse problem can be applied, see for example \cite{RW2017stochastic}.}, is very complicated since one needs to characterize the order of the decision process explicitly.
For ease of analysis, we assume that there exists a multiplier $\lambda_j\in\mathfrak{L}^1_1$ corresponds to second-stage constraint and study the saddle-point condition of the Lagrangian formulation of problem (\ref{twostage}).
It is worth mentioning that the Karush-Kuhn-Tucker (KKT) condition of problem (\ref{twostage}) (see \cite{RW1975stochastic}) introduces a second-stage multiplier $\tilde{\lambda}_j\in(\mathfrak{L}^\infty_1)^*$ for every $j\in\mathcal{J}$ which incorporates the two-stage decision making process.
This can be seen from the fact that any element of the dual space $(\mathfrak{L}^\infty_1)^*$ can be decomposed into a component of $\mathfrak{L}^1_1$ and a ``singular'' component, corresponding to the multiplier of nonanticipativity.
The saddle-point condition is shown to be sufficient and ``almost'' necessary for optimality of problem (\ref{twostage}), and we refer the interested readers to \cite{RW1975stochastic,RW1976stochastic,RW1976stochastic1,RW1976stochastic2,RW2017stochastic} for more details.

The Lagrangian formulation of problem (\ref{twostage}) associated with agent $j$ is of the following form
\[
L_j(x_j,x^*_{-j}, y_j, y^*_{-j}, \lambda_j) = L_j^1(x_j,x^*_{-j}) + \mathbb{E}\big[L_j^2\big(\xi,x_j,y_j(\xi),y^*_{-j}(\xi),\lambda_j(\xi)\big)\big],
\]
where
\[\begin{aligned}
&L_j^1(x_j,x^*_{-j}) = W^1_j(x_j,x^*_{-j}),\\
&L_j^2\big(\xi,x_j,y_j(\xi),y^*_{-j}(\xi),\lambda_j(\xi)\big) = W^2_j(\xi, y_j(\xi),y^*_{-j}(\xi)) + \lambda_j(\xi)(x_j-y_j(\xi)).
\end{aligned}\]
The constraints ${y_j(\xi)}\leq_{a.s.} x_j$ can be interpreted as the situation under which the profit maximizing supply $y^*_j(\xi)$ of agent $j$ is not necessarily equal to the total production quantity $x_j$.
This feature of our model differs from conventional requirement on production-clearing condition, i.e., all the produced goods are expected to supply to the market.

In order to make further progress in characterizing the CN equilibrium, we need to specify the structures of our wealth functions, $W^1_j$ and $W^2_j$, suitable for our application.
We assume that the production cost for $j$-th agent is quadratic, i.e., for each $j\in\mathcal{J}$ the cost of producing $x_j$ amount of production is $\frac{1}{2}c_jx^2_j+a_jx_j$, for some $c_j>0, a_j>0$.
In the second-stage, the cost function of the supply or second-stage is assumed to be linear and of stochastic nature, i.e., for each $j\in\mathcal{J}$ the cost of supplying ${y_j(\xi)}$ amount of commodity is ${h_j(\xi)} {y_j(\xi)}$ for a.e. $\xi\in\Xi$. Here, $h_j(\xi)$ can be regarded as the  unite supply cost of agent $j$.
We adopt a classic stochastic inverse demand curve, see for example \cite{HLP2013demand}, that takes the expression $p\big(\xi, T(y(\xi))\big)={p_0}(\xi)-\gamma(\xi) T(y(\xi))$ for the spot price.
In practice, the stochastic \emph{benchmark price} excluding the effect of supply to the market $p_0:\Xi\mapsto\mathbb{R}_{+}$ can be estimated via statistical approaches based on real data.
The supply discount $\gamma:\Xi\rightarrow \mathbb{R}_{+}$ acts as a market mechanism to adjust and reflect uncertainty in quantity in the market.
In order to respect the market mechanism of supply-demand relation, we make the following assumption through out our study.
\begin{assumption}\label{gamma0}
There exists a $\gamma_0>0$ such that $\gamma(\xi)\geq\gamma_0$ for a.e. $\xi\in\Xi$.
\end{assumption}
Thus, agent $j$'s stage-wise wealth functions are,
\[
W^1_j(x_j,x^*_{-j}) = -\frac{1}{2}c_jx_j^2-a_jx_j,
\]
and
\[
W^2_j(\xi, y_j(\xi),y_{-j}(\xi)) = \big(p_j(\xi)-\gamma(\xi)T(y(\xi))\big)y_j(\xi),
\]
where the short-handed notation of the risk-adjusted spot price of agent $j$'s is denoted by $p_j(\xi):= p_0(\xi)-h_j(\xi).$

We are now ready to consider the specific stochastic programming for every agent $j\in \mathcal{J}$ {\color{blue}{:}}
\begin{equation}
\label{M1}
\begin{aligned}
&\text{maximize}_{x_j}~&&\mathbb{E}[\Phi_j(\xi,x)] -\frac{1}{2}c_jx_j^2-a_jx_j\\
&\text{subject to}~&&0\leq x_j,
\end{aligned}
\end{equation}
where
\begin{equation}
\label{M2}
\begin{aligned}
&\Phi_j(\xi,x)=&&\text{maximize}_{y_j(\xi)}&&\Big({p_j(\xi)}-\gamma(\xi) \Big(\sum^{J}_{i\neq j}{y_i^*(\xi)}+{y_j(\xi)} \Big)\Big){y_j(\xi)}\\
&~&&\text{subject to}&&0\leq {y_j(\xi)}\leq x_j, ~ \text{for a.e. $\xi\in \Xi$}.
\end{aligned}
\end{equation}
Note that the requirements in problem (\ref{M2}) hold a.s. in accordance with the a.s. constraints of the second-stage in problem (\ref{twostage}).
However, \eqref{M1}-\eqref{M2} is not easy to solve, especially in a SNEP with $J\geq 2$, see \cite{CSS2018convergence}.
The complication arises since the $j$-th agent's problem contains that of the other agents' strategy, not yet known at the decision horizon.
A commonly used method is to recast problem \eqref{M1}-\eqref{M2} of each agent as a stochastic equilibrium problem.
Then, obtaining an equilibrium of the convex $J$-player game \eqref{M1}-\eqref{M2} is equivalent to finding its solution for all agents.
Stochastic equilibrium has been shown to be an effective method to study and to solve two-stage multi-players stochastic game problems, see for instance \cite{CSX2017discrete,PSS2017two,SSC2017saa,RW2017stochastic,RS2018solving}.
We study the saddle-point condition of the problem (\ref{M1})-(\ref{M2}), rewritten in the form of problem (\ref{twostage}).
More specifically, for all $j\in\mathcal{J}$, there exists $\bar{\lambda}_j(\xi)\in\mathfrak{L}^1_1$ with ${\bar{\lambda}(\xi)}\geq_{a.s.}0$ so that a strategy $\big(\bar{x}_j, \bar{y}_j\big)\in\mathbb{R}_+\times \mathfrak{L}^\infty_+$ solves the following system.
\[\begin{aligned}
& -c_j\bar{x}_j-a_j + \mathbb{E}\big[\bar{\lambda}_j(\xi)\big] \in \mathcal{N}_{[0,\infty)}(\bar{x}_j),\\
& {p_j(\xi)} - \gamma(\xi)\sum^J_{i\neq j} \bar{y}_i(\xi) - 2\gamma(\xi)  \bar{y}_j(\xi) - \bar{\lambda}_j(\xi)\in_{a.s.} \mathcal{N}_{[0,\infty)}(\bar{y}_j(\xi)), &&\text{(stationality)}\\
&\bar{x}_j\geq0, {\bar{y}}_j(\xi)\geq_{a.s.}0, \bar{x}_j-{\bar{y}}_j(\xi)\geq_{a.s.}0, &&\text{(feasibility)}\\
&\bar{\lambda}_j(\xi)\geq_{a.s.}0, &&\text{(dual feasibilty)}\\
&\bar{\lambda}_j(\xi)\bot_{a.s.}(\bar{x}_j-{\bar{y}}_j(\xi)). &&\text{(complementarity)}
\end{aligned}\]
In particular, \emph{stationarity} comes from the first-order necessary optimality condition under the assertion $\partial\mathbb{E}\big[\Phi_j(\xi,x)\big]\subseteq\mathbb{E}\big[\partial_x\Phi(\xi,x)\big]$.
The assertion is discussed in \cite{BCS2017subdifferentiation}, and the above system can be viewed as a weaker condition for optimality.

Rewritten in a compact form as SVI, the optimal strategy-multiplier pair $(x_j,y_j,\lambda_j)\in\mathbb{R}_+\times\mathfrak{L}^\infty_+\times\mathfrak{L}^1_+$ must satisfy,
\begin{equation}\label{twostageSVI}\begin{aligned}
&0\leq ~~~~x_j&&\bot~~~~~c_jx_j+a_j-\mathbb{E}\big[ \lambda_j(\xi)\big] &&\geq ~~~~0,\\
&0\leq_{a.s.} y_j(\xi)&&\bot_{a.s.} -p_j(\xi) + \gamma(\xi)\sum^J_{i\neq j} y_i(\xi) + 2\gamma(\xi) y_j(\xi) + \lambda_j(\xi) &&\geq_{a.s.}0,\\
&0\leq_{a.s.} \lambda_j(\xi)&&\bot_{a.s.}~x_j-y_j(\xi)&&\geq_{a.s.}0.
\end{aligned}\end{equation}

It follows that since all agents in oligopolistic market act non-cooperatively, we write down the equilibrium interpreted as that of the whole system.
More specifically, let $x=(x_1,\ldots,x_J)^T$ be the first-stage decision vectors of the system, and for a.e. $\xi\in\Xi$, the second-stage decision vector $y(\xi)=\big(y_1(\xi),\ldots,y_J(\xi) \big)^T$ and its corresponding multiplier vector $\lambda(\xi)=\big(\lambda_1(\xi),\ldots,\lambda_J(\xi)\big)^T$ are denoted respectively.
Similarly, parameter vectors can be written collectively as $a=(a_1,\ldots,a_J)^T$, $p(\xi)=(p_1(\xi),\ldots,p_J(\xi))^T$.  Then, we can treat the SVI for all agents as a two-stage stochastic complementary problem (SCP):
\begin{equation}
\label{SEQ}
\begin{aligned}
0&\leq x &&\bot ~Cx-  \mathbb{E}[\lambda(\xi)] + a &&\geq 0,\\
0&\leq
\begin{pmatrix}
y(\xi)\\
\lambda(\xi)
\end{pmatrix}
&&\bot
\begin{pmatrix}
\Pi(\xi) & I\\
-I & 0
\end{pmatrix}
\begin{pmatrix}
y(\xi)\\
\lambda(\xi)
\end{pmatrix}
+
\begin{pmatrix}
-p(\xi)\\
x
\end{pmatrix}
&&\geq 0, \quad \text{for a.e. $\xi \in \Xi$,}
\end{aligned}
\end{equation}
where
\[
C=\text{diag}(c_1,c_2,...,c_J),\quad \Pi(\xi)=\gamma(\xi)(ee^T+I).
\]
It follows that for the whole system, a $J$-tuple of strategies $$\big(x^*,y^*,\lambda^*\big) = \big((x_1^*,y_1^*,\lambda_1^*),\ldots,(x_J^*,y_J^*,\lambda_J^*)\big)\in\mathbb{R}^J\times\mathfrak{L}^\infty_J\times\mathfrak{L}^1_J$$ is called a solution of the two-stage SCP (\ref{SEQ}).

\section{Structure of the regularized two-stage SCP}\label{sec:2}
In this section, we focus on characterizing solutions of two-stage stochastic linear complementarity problem \eqref{SEQ}.
From the derivation of first-order necessary optimality conditions of problem (\ref{M1})-(\ref{M2}) and the monotonicity of problem \eqref{SEQ}, we have the following results on existence of solutions.
\begin{proposition}[Theorem 2, \cite{R1965existence}]
\label{p1}
For any fixed pair $(x,\xi)\in \mathbb{R}_+^J\times \Xi$, the second-stage problem \eqref{M2} has a unique solution.
\end{proposition}

Thus, for the two-stage stochastic linear complementarity problem (\ref{SEQ}), the following proposition holds.
\begin{proposition}
\label{p7}
The two-stage stochastic linear complementarity problem \eqref{SEQ} has relatively complete recourse, i.e., for any $x\in \mathbb{R}_+^J$ and a.e. $\xi\in \Xi$ the second-stage problem  of \eqref{SEQ} is solvable.
\end{proposition}
\begin{proof}
The coefficient matrix of the second-stage part of (\ref{SEQ})
$$
M(\xi)=\begin{pmatrix}
\Pi(\xi) & I\\
-I & 0
\end{pmatrix}
$$
is positive semidefinite for a.e. $\xi\in \Xi$.

For any given $x\in \mathbb{R}_+^J$, it follows that there always exists a pair $\big(\hat{y}(\xi),\hat{\lambda}(\xi)\big)\in \mathbb{R}^J\times\mathbb{R}^J$, such that
\begin{equation*}
  \begin{pmatrix}
    \hat{y}(\xi)\\
    \hat{\lambda}(\xi)
  \end{pmatrix}
  \geq 0,~
  \begin{pmatrix}
    \Pi(\xi) & I\\
    -I & 0
  \end{pmatrix}
  \begin{pmatrix}
    \hat{y}(\xi)\\
    \hat{\lambda}(x)
  \end{pmatrix}
  +
  \begin{pmatrix}
    -p(\xi)\\
    x
  \end{pmatrix}
  \geq 0, \quad \text{ for a.e. $\xi\in \Xi$}.
\end{equation*}
In detail, we consider a special choice $\hat{y}(\xi)=0$ and $\hat{\lambda}(\xi)=\max\{0,p(\xi)\}$, where the $\max$ function is taken componentwise. Thus, the corresponding quadratic programming problem of the linear complementarity problem is feasible.  It follows from \cite[Lemma 3.1.1, Theorem 3.1.2]{RPS1992the} that there must exist at least a solution which solves the second-stage problem for any given pair $(x,\xi)$.
\qed\end{proof}

Although the second-stage problem \eqref{M2} has a unique equilibrium for any given $(x,\xi)$ (see Proposition \ref{p1}), the system \eqref{SEQ} may admit multiple solutions. To see this, we give an illustrative example.
\begin{example}
\label{eg1}
Consider a duopoly game, with given $x=(x_1,x_2)^T\geq 0$, and $-p(\xi)\geq_{a.s.} 0$.
Then, the corresponding second-stage part of complementarity system \eqref{SEQ} reads
\begin{equation}
\label{gs1}
0\leq
\begin{pmatrix}
y_1(\xi)\\
y_2(\xi)\\
\lambda_1(\xi)\\
\lambda_2(\xi)
\end{pmatrix}
\bot
\begin{pmatrix}
2\gamma(\xi) & \gamma(\xi) & 1 & 0\\
\gamma(\xi) & 2\gamma(\xi) & 0 & 1\\
-1 & 0 & 0 & 0\\
0 & -1 & 0 & 0\\
\end{pmatrix}
\begin{pmatrix}
y_1(\xi)\\
y_2(\xi)\\
\lambda_1(\xi)\\
\lambda_2(\xi)
\end{pmatrix}
+
\begin{pmatrix}
-p_1(\xi) \\
-p_2(\xi) \\
x_1\\
x_2
\end{pmatrix}
\geq 0,  \quad \text{for a.e. $\xi\in \Xi$}.
\end{equation}
Then, the solution set of \eqref{gs1} is of the following form
\begin{align*}
\left\{(0,0,\tilde{\lambda}_1(\xi),\tilde{\lambda}_2(\xi)):
\tilde{\lambda}_1(\xi)=
\begin{cases}
0,&x_1>0\\
\lambda_1(\xi),&x_1=0
\end{cases},~
\tilde{\lambda}_2(\xi)=
\begin{cases}
0,&x_2>0\\
\lambda_2(\xi),& x_2=0 \\
\end{cases}, \quad \text{for a.e. $\xi\in \Xi$}
\right\},
\end{align*}
where $\lambda_1(\xi)\geq_{a.e.}0, \lambda_2(\xi)\geq_{a.e.}0$.
\end{example}
In Example \ref{eg1}, the ``equilibrium price'' $\lambda$ may admit multiple values when there exist some zero-valued components of $x$.

Technically, the multiple solutions of the second-stage problem will cause trouble when we handle the two-stage stochastic complementarity system \eqref{SEQ}, both in computation and analysis \cite{{SDR2009lectures}}.
The assumption ensuring the uniqueness of second-stage solution is usually made, see for instance \cite{CSS2018convergence,CSX2017discrete}.
Moreover, interpreted as ``equilibrium price'' associated with agents' production clearing, different values of $\lambda$ would have ambiguous economical interpretations.
Motivated by these, we propose a regularized method to seek for one particular choice of ``equilibrium price''.
Similar approach can be found in for example \cite{CSW2015regularized}.

For an $\epsilon>0$, let
\begin{equation*}
M^\epsilon(\xi)=
\begin{pmatrix}
\Pi(\xi) & I\\
-I & \epsilon I
\end{pmatrix}~
\mathrm{and}~
q(x,\xi)=
\begin{pmatrix}
-p(\xi) \\
x
\end{pmatrix}.
\end{equation*}
Thus, we propose
the regularized SCP of \eqref{SEQ} as follows:
\begin{equation}
\label{RM}
\begin{aligned}
0&\leq x &&\bot~Cx-  \mathbb{E}[\lambda(\xi)] + a &&\geq 0,\\
0&\leq
\begin{pmatrix}
y(\xi)\\
\lambda(\xi)
\end{pmatrix}
&&\bot
~\begin{pmatrix}
\Pi(\xi) & I\\
-I & \epsilon I
\end{pmatrix}
\begin{pmatrix}
y(\xi)\\
\lambda(\xi)
\end{pmatrix}
+
\begin{pmatrix}
-p(\xi) \\
x
\end{pmatrix}
&&\geq 0, \quad \text{for a.e. $\xi\in \Xi$}.
\end{aligned}
\end{equation}

For a given pair $(x,\xi)\in\mathbb{R}_+^J\times\Xi$, the second-stage problem of \eqref{SEQ} and the regularized second-stage problem \eqref{RM} are denoted by LCP$(q(x,\xi),M(\xi))$ and LCP$(q(x,\xi),M^\epsilon(\xi))$ respectively.  Their solution functions are chosen from the respective solution sets and expressed by $z(q(x,\xi))$ and $z^\epsilon(q(x,\xi))$.
In the sequel, we omit the $\xi$ and $x$ without causing confusion, i.e., LCP$(q,M):=\text{LCP}(q(x,\xi),M(\xi))$ and LCP$(q,M^\epsilon):=\text{LCP}(q (x,\xi),M^\epsilon(\xi))$.

For clearer demonstration, recall our illustrative Example \ref{eg1}, and consider its regularization approach.
Thus, the second-stage of the regularized problem takes the following form
\begin{equation}
\label{zgs2}
0\leq
\begin{pmatrix}
y_1(\xi)\\
y_2(\xi)\\
\lambda_1(\xi)\\
\lambda_2(\xi)
\end{pmatrix}
\bot
\begin{pmatrix}
2\gamma(\xi) & \gamma(\xi) & 1 & 0\\
\gamma(\xi) & 2\gamma(\xi) & 0 & 1\\
-1 & 0 & \epsilon & 0\\
0 & -1 & 0 & \epsilon\\
\end{pmatrix}
\begin{pmatrix}
y_1(\xi)\\
y_2(\xi)\\
\lambda_1(\xi)\\
\lambda_2(\xi)
\end{pmatrix}
+
\begin{pmatrix}
-p_1(\xi) \\
-p_2(\xi) \\
x_1\\
x_2
\end{pmatrix}
\geq 0,  \quad \text{for a.e. $\xi\in \Xi$}.
\end{equation}
Under the same condition as in Example \ref{eg1}, we can obtain the unique solution of \eqref{zgs2}, which $\tilde{y}_1,\tilde{y}_2,\tilde{\lambda}_1,\tilde{\lambda}_2$  equal to 0 for a.e. $\xi\in \Xi$.
Due to the positive definiteness of $C$, it follows that we obtain the unique solution of the first-stage problem is $x_1=0,x_2=0$.
Then, we have obtained one particular solution of the original problem, the trivial solution in this example.
The key feature of our regularized method is that it promises the existence and uniqueness of solution due to the strongly monotone of regularized two-stage problem.

In the remaining of this section, we concern ourselves with the solution $z^\epsilon$ of LCP$(q,M^\epsilon)$ and explore the structure of the second-stage solution.
\begin{proposition}
\label{p3}
For any fixed $\epsilon>0$, the regularized problem \eqref{RM} has a unique solution $(x^\epsilon, y^\epsilon, \lambda^\epsilon)\in \mathbb{R}^J\times\mathcal{Y}\times\mathcal{Y}$.
\end{proposition}
\begin{proof}
The result can be obtained via a similar procedure as in \cite[Proposition 2.1 (i)]{CSX2017discrete} and we only need to show that the condition in \cite[Assumption 1]{CSX2017discrete} holds.
Recall that Assumption \ref{gamma0} holds, then for a.e. $\xi\in\Xi$
\begin{equation*}
 \left(
    \begin{array}{c}
      x \\
      u(\xi) \\
      v(\xi) \\
    \end{array}
  \right)^T
  \begin{pmatrix}
    C & 0 & -I\\
    0 & \Pi(\xi) &I\\
    I & -I & \epsilon I
  \end{pmatrix}
  \left(
    \begin{array}{c}
      x \\
      u(\xi) \\
      v(\xi) \\
    \end{array}
  \right)\geq \tau(\|x\|^2+\|u(\xi)\|^2+\|v(\xi)\|^2),
\end{equation*}
 where $\tau=\min\{\bar{c},\gamma_0(J+1),\epsilon\}$ with $\bar{c}$ denoting the minimum diagonal element of $C$.
\qed\end{proof}

\begin{theorem}\label{th:cf}
For any fixed $\epsilon>0$, $x\geq 0$ and a.e. $\xi\in \Xi$, the $j$-th component of the second-stage solution of problem \eqref{RM} $\big((y^\epsilon)_j,(\lambda^\epsilon)_j\big)$ is either $(0,0)$, or one of the following two forms:
\begin{equation}
\label{CFS}
\begin{aligned}
&\displaystyle -\left(\frac{\gamma(\xi)T^\epsilon - p_j(\xi)}{\gamma(\xi)}, \quad \quad 0\right),\\
&\displaystyle -\left(\frac{\epsilon(\gamma(\xi)T^\epsilon -p_j(\xi)) - x_j}{\epsilon\gamma(\xi) +1}, \quad \frac{\gamma(\xi)(T^\epsilon+x_j)- p_j(\xi)}{\epsilon\gamma(\xi) +1} \right)
\end{aligned}
\end{equation}
for $j\in\mathcal{J}$, where
\begin{equation}
\label{gs6}
 T^\epsilon:=\sum_{i=1}^J(y^\epsilon)_i = \frac{\gamma(\xi)\sum_{i\in\mathcal{I}_3}x_i + \epsilon\gamma(\xi)\sum_{i\in\mathcal{I}_2\cup\mathcal{I}_3}p_i(\xi)
+ \sum_{i\in\mathcal{I}_2}p_i(\xi)}{\big(\epsilon\gamma(\xi)(\abs{\mathcal{I}_2}+\abs{\mathcal{I}_3}+1) + \abs{\mathcal{I}_2} + 1 \big) \gamma(\xi)}
\end{equation}
with
\begin{align*}
\mathcal{I}_2&=\left\{j\in\mathcal{J}: \gamma(\xi)T^\epsilon + (\lambda^\epsilon)_j - p_j(\xi) < 0, \quad (y^\epsilon)_j - x_j \leq 0 \right\},\\
\mathcal{I}_3&=\left\{j\in\mathcal{J}: \gamma(\xi)T^\epsilon + (\lambda^\epsilon)_j - p_j(\xi) < 0, \quad (y^\epsilon)_j - x_j   >   0 \right\},
\end{align*}
where
$|\mathcal{I}_2|$ and $|\mathcal{I}_3|$ denote the cardinality of $\mathcal{I}_2$ and $\mathcal{I}_3$ respectively.
\end{theorem}

The detailed proof of the Theorem \ref{th:cf} is given in Appendix.
Note that the above theorem gives the forms of the unique solution of the second-stage regularized problem \eqref{RM}.
However, it may be used to assist numerical calculation since the partition of the index set is not known in advance.
Nevertheless, it is suffice for our purposes of deriving additional properties of the solutions.
Due to the positive definiteness of $M^\epsilon$ and special structure of problem (\ref{RM}), we first obtain the following Lipschitz continuous property, following \cite[Corollary 2.1]{CX2013newton}.
\begin{lemma}
\label{lem3}
For any $\xi\in\Xi$, there exists $L(\xi)>0$ such that for any fixed $\epsilon \in (0,1]$, we have
\[\|z^\epsilon(q(x_1,\xi))-z^\epsilon(q(x_2,\xi))\|\leq L(\xi)\|x_1-x_2\|,\,\,
\text{~for ~}x_1,x_2\in\mathbb{R}^J_+.\]
\end{lemma}

\begin{lemma}
\label{lem2}
For any fixed $\epsilon>0$ and $(x,\xi)\in\mathbb{R}_+^J\times\Xi$, $T^\epsilon$ has the following upper bound:
\begin{align*}
T^\epsilon&\leq \norm{x}_1 + \left(\epsilon + \frac{1}{\gamma(\xi)}\right)\norm{p(\xi)}_1.
\end{align*}
\end{lemma}
\begin{proof}
We have the following derivation from \eqref{gs6} that
\begin{align*}
T^\epsilon= &\frac{\gamma(\xi)\sum_{i\in\mathcal{I}_3}x_i +  \epsilon\gamma(\xi)\sum_{i\in\mathcal{I}_2\cup\mathcal{I}_3}{p_i(\xi)}
+ \sum_{i\in\mathcal{I}_2}p_i(\xi)}{\left(\epsilon\gamma(\xi)(\abs{\mathcal{I}_2}+\abs{\mathcal{I}_3}+1) + \abs{\mathcal{I}_2} + 1 \right) \gamma(\xi)}\\
\leq & \frac{\gamma(\xi)\sum_{i=1}^Jx_i + (\epsilon\gamma(\xi) +1)\sum_{i=1}^J\abs{p_i(\xi)}
}{\left(\epsilon\gamma(\xi)(\abs{\mathcal{I}_2}+\abs{\mathcal{I}_3}+1) + \abs{\mathcal{I}_2} + 1 \right) \gamma(\xi)}\\
\leq & \frac{\gamma(\xi)\sum_{i=1}^Jx_i + (\epsilon\gamma(\xi)+1)\sum_{i=1}^J\abs{p_i(\xi)}
}{\gamma(\xi)}\\
=& \norm{x}_1 + \left(\epsilon + \frac{1}{\gamma(\xi)}\right)\norm{p_1(\xi)}_1.
\end{align*}
\qed\end{proof}

We end this section by establishing the convergence result of the second-stage LCP$(q,M^\epsilon)$ solutions as $\epsilon\downarrow0$ for any given pair $(x,\xi)\in\mathbb{R}^J\times\Xi$.
\begin{proposition}
\label{p2}
For any fixed $\epsilon>0$ and $(x,\xi)\in\mathbb{R}_+^J\times\Xi$, let $z^\epsilon(\xi)=(y^\epsilon(\xi),\lambda^\epsilon(\xi))$ denote the unique solution of the regularized problem LCP$(q,M^\epsilon)$. Then
$$\lim_{\epsilon\downarrow 0}\norm{z^\epsilon(\xi)-\bar{z}(\xi)}=0,$$
where $\bar{z}(\xi)=(\bar{y}(\xi),\bar{\lambda}(\xi))$ denotes the unique least $l_2$-norm solution of the LCP$(q,M)$.
Moreover, the $j$-th component of the least $l_2$-norm solution of problem (\ref{SEQ}) has one of the following three forms:
\begin{equation}
\label{LNS}
\left\{(0,0), \left(-\frac{\gamma(\xi)\bar{T} - p_j(\xi)}{\gamma(\xi)},0\right), \Big( x_j, -\gamma(\xi)(\bar{T}+x_j) + p_j(\xi) \Big)\right\}
\end{equation}
for $j\in\mathcal{J}$, where
\begin{align*}
\bar{T} := \lim_{\epsilon\downarrow0}T^\epsilon= \sum_{i=1}^{J}\bar{y}_i.
\end{align*}
Furthermore, for a.e. $\xi\in\Xi$ there exists $\bar{\kappa}(\xi)>0$, such that
\begin{equation}
\label{gs7}
\Vert \lambda^\epsilon(\xi) - \bar{\lambda}(\xi)\Vert \leq \bar{\kappa}(\xi)\epsilon.
\end{equation}
\end{proposition}
The proof of Proposition \ref{p2} is given in Appendix.

\section{Convergence analysis}
\label{sec:3}
In this section, we first prove the convergence of the unique solution of the regularized problem \eqref{RM}
to the solution set of the original problem as the regularized parameter $\epsilon$ decreases to zero.
Next, we will study the sample average approximation (SAA) to solve the regularized problem. See \cite{CSW2015regularized}.
Combined with our regularization approaches, we demonstrate the convergence property of the solution of our regularized SAA model as the number of samples goes to infinity.
More specifically, the convergence analysis in this section is divided into two parts: the convergence analysis of the regularized problem as the regularized parameter $\epsilon$ tends to zero, and the analysis of regularized SAA.
We finally build up the convergence relationship between the regularized SAA approach and the original problem.

\subsection{Convergence of the regularized model}
In this subsection, we only need to consider the convergence properties of the first-stage decision vector, i.e., $x^\epsilon\in\mathbb{R}^J_+$ that solves problem (\ref{RM}),
when the regularized parameter $\epsilon$ tends to zero.
The convergence property of the solution $(x^\epsilon,y^\epsilon,\lambda^\epsilon)$ then follows combining the result of section 3.
From Proposition \ref{p3}, we know that for fixed $\epsilon>0$ problem \eqref{RM} admits a unique first-stage solution $x^\epsilon$. In the following, we concern about the sequence of accumulation points of $\{x^\epsilon\}$ as $\epsilon\downarrow0$.

For the existence of accumulation points, we have the following result.
\begin{proposition}
\label{p4}
Suppose there exists $p_0>0$ such that for all $j\in\mathcal{J}$, $p_j(\xi)\leq_{a.s.} p_0$. Then, with $\epsilon\downarrow0$, $\{x^\epsilon\}$ is bounded.
\end{proposition}
\begin{proof}
From the condition on $p_j(\xi)$, there must a sufficiently large $\alpha>0$ such that for any $j\in\mathcal{J}$
$$\gamma_0\alpha-p_j(\xi)>0, \quad \text{for a.e. $\xi\in \Xi$}.$$
Then we have
$$-\frac{\gamma(\xi)(T^\epsilon+\alpha) - {p_j(\xi)}}{\epsilon\gamma(\xi)+1} < 0,\quad \text{for a.e. $\xi\in \Xi$}.$$
Assume that $\{x^\epsilon\}$ is unbounded for the purpose of arriving at a contradiction.
Then, it follows that there exist some indices $j\in\mathcal{J}$, such that $(x^\epsilon)_j\geq\alpha$. Then we consider the $j$-th component of the first-stage complementarity relation,
$$0 \leq (x^\epsilon)_j \bot c_j(x^\epsilon)_j - \mathbb{E}\big[\big(\lambda^\epsilon(\xi)\big)_j\big] + a_j \geq 0,$$
which can be expressed, from \eqref{CFS}, as
$$0 \leq (x^\epsilon)_j \bot c_j(x^\epsilon)_j +  a_j \geq 0.$$
However, this complementarity relation cannot be obtained because $(x^\epsilon)_j>0$ and $c_j(x^\epsilon)_j +  a_j>0$.
This completes our proof.
\qed\end{proof}

Note that the conditions $p_j(\xi)\leq_{a.s.} p_0$ can be easily satisfied in many practical applications.
For example, with given data sets of $p(\xi)$ we can always find an upper bound $p_0:=\max_j\{p_j(\xi)\}$.
\begin{lemma}
\label{lem1}
Suppose there exists a constant $p_0>0$ such that for all $j\in\mathcal{J}$, $p_j(\xi)\leq_{a.s.} p_0$.
Then, there exists a sequence $\{\epsilon_k\}_{k=1}^\infty$ with $\epsilon_k\downarrow0$ as $k\rightarrow\infty$ such that $x^k\to\hat{x}$ and $\lambda^k\to \hat{\lambda}$,
$$\lim_{k\rightarrow\infty}\mathbb{E}[\lambda^{k}(\xi)] = \mathbb{E}[\bar{\lambda}(\xi)],$$
where $x^k$ and $\lambda^k(\xi)$ are parts of the unique solution of problem \eqref{RM} with $\epsilon=\epsilon_k$, and $\bar{\lambda}(\xi)$ is part of the least norm solution of the second-stage problem \eqref{RM} with $x=\hat{x}$.
\end{lemma}

\begin{proof}
From Proposition \ref{p4}, there exists a sequence $\{x^k\}$ such that $\lim_{k\to\infty}x^k=\hat{x}$.
From \eqref{gs7}, we have
\begin{align*}
&~~~~\norm{\mathbb{E}[\hat{\lambda}(\xi)]- \mathbb{E}[\bar{\lambda}(\xi)]}\\
&\leq \norm{\mathbb{E}[\lambda^k(\xi)] - \mathbb{E}[\bar{\lambda}(\xi)]} + \norm{\mathbb{E}[\hat{\lambda}(\xi)]- \mathbb{E}[\lambda^k(\xi)]} \\
&\leq\mathbb{E}[\bar{\kappa}(\xi)]\epsilon_k + \norm{\mathbb{E}[\hat{\lambda}(\xi)]- \mathbb{E}[\lambda^k(\xi)]}.
\end{align*}
From Lemma \ref{lem3} and \eqref{gs7}, for a.e. $\xi\in\Xi$
$$\norm{\lambda^k(\xi) - \hat{\lambda}(\xi)}\rightarrow 0~\text{as}~k\rightarrow \infty.$$
We have from the derivation that the estimation
\begin{align*}
\norm{\lambda^k(\xi) - \hat{\lambda}(\xi)} &\leq \norm{\lambda^k(\xi)} + \norm{\hat{\lambda}(\xi)}\\
&\leq 4\sqrt{J}\big(\gamma(\xi)\norm{w}_1 + \norm{p(\xi)}_1\big), \quad \text{for a.e. $\xi\in\Xi$},
\end{align*}
where the last term comes from Lemma \ref{lem3} with some vector $w$ with $\{x^k\}_{k=1}^\infty, \hat{x}\subset [0, w]$.
It follows from the Lebesgue Dominated Convergence Theorem, we have
\begin{align*}
\label{gs8}
\norm{\mathbb{E}[\lambda^k(\xi)] - \mathbb{E}[\bar{\lambda}(\xi)]}\rightarrow 0 ~\text{as}~k\rightarrow \infty.
\end{align*}
Then we complete the proof.
\qed\end{proof}

\begin{theorem}
\label{Th3}
Any accumulation point of $\{x^\epsilon,y^\epsilon,\lambda^\epsilon\}$ as $\epsilon\downarrow0$ is a solution of problem \eqref{SEQ}.
\end{theorem}

\begin{proof}
We only need to verify that for any $\epsilon_k\downarrow0$, the accumulation point $\hat{x}$ of subsequence $\{x^{\epsilon_k}\}$ is a first-stage solution of \eqref{SEQ}.
Since $x^{\epsilon_k}$ is the first-stage solution of problem \eqref{RM} for any $\epsilon_k>0$, we have with $x^k=x^{\epsilon_k}$
$$0\leq x^{k} \bot Cx^{k}  -  \mathbb{E}[{\lambda}^{k}(\xi)] + a \geq 0,$$
which, by using the `min' NCP function (see, for example, \cite{RPS1992the}), can be rewritten as
$$\min\{ x^{k}, Cx^{k}  -  \mathbb{E}[{\lambda}^{k}(\xi)] + a \} =0.$$
By Lemma \ref{lem1}, we have
$$0=\lim_{k\rightarrow\infty}\min\{ x^{k} , Cx^{k} -  \mathbb{E}\big[{\lambda}^{k}(\xi)\big] + a \} = \min\{ \hat{x} , C\hat{x}  -  \mathbb{E}[\bar{\lambda}(\xi)] + a \}$$
as $k\rightarrow\infty$. Thus we obtain that
$$\min\{ \hat{x} , C\hat{x}  -  \mathbb{E}[\bar{\lambda}(\xi)] + a \} =0.$$
The statement then follows from Proposition \ref{p2}.
\qed\end{proof}

\subsection{Convergence of the regularized SAA model}
In this subsection, we study the SAA scheme for solving the regularized problem \eqref{RM} and focus on the convergence of the regularized SAA approach.
More specifically, we focus on the SAA convergence analysis and the solution of the first-stage problem.
It is noteworthy that Chen, Sun and Xu considered a discrete approximation scheme in \cite{CSX2017discrete}, which also leads to an approximation of the response variable in the second-stage problem.

Let $\xi_1,\xi_2,\ldots,\xi_\nu$ denote $\nu$ independent identically distributed (i.i.d.) samples.
Then, with slight abuse of notation, we can obtain the following formulation of problem \eqref{RM} with SAA:
\begin{equation}
\label{gs19}
\begin{aligned}
&0\leq x \bot Cx-  \frac{1}{\nu}\sum_{\ell=1}^\nu\lambda(\xi_\ell) + a \geq 0,\\
&0\leq
\begin{pmatrix}
y(\xi_\ell)\\
\lambda(\xi_\ell)
\end{pmatrix}
\bot
\left(
  \begin{array}{cc}
    \Pi(\xi_\ell) & I \\
    -I & \epsilon I \\
  \end{array}
\right)
\begin{pmatrix}
y(\xi_\ell)\\
\lambda(\xi_\ell)
\end{pmatrix}
+
\left(
  \begin{array}{c}
    -p(\xi_\ell) \\
    x \\
  \end{array}
\right)
\geq 0, \,\ell=1,\ldots,\nu.
\end{aligned}
\end{equation}
Or, we can write the problem collectively for all $\nu$ samples
\begin{equation}\label{SLEQ-matrix}
 0 \leq  \left(
 \begin{array}{c}
            x \\
            v_1 \\
            \vdots \\
            v_\nu \\
          \end{array}
        \right)
 \bot
  \left(\begin{array}{cccc}
          C & \frac{1}{\nu} B & \ldots & \frac{1}{\nu}B \\
          -B^T & D^\epsilon_1 &  &  \\
          \vdots &  & \ddots &  \\
          -B^T &  &  & D^\epsilon_\nu \\
        \end{array}
      \right)
      \left(
          \begin{array}{c}
            x \\
            v_1 \\
            \vdots \\
            v_\nu \\
          \end{array}
        \right)
     + \left(
      \begin{array}{c}
          a \\
          q_1 \\
          \vdots \\
          q_\nu \\
        \end{array}
      \right)\geq 0,
\end{equation}
where
$C \in R^{J\times J},$
      $B= \left(\begin{array}{cc}
          0 & -I \\
        \end{array}
      \right)
\in R^{J\times 2J}   $, and for $\ell=1,\ldots,\nu$,
      $D^\epsilon_\ell= \left(\begin{array}{cccc}
          \Pi(\xi_\ell) & I \\
          -I & \epsilon I \\
        \end{array}
      \right) \in R^{2J \times 2J}  $, $ v_\ell=(y(\xi_\ell),$ $\lambda(\xi_\ell))^T,  q_{\ell}=(-p(\xi_\ell),0)^T.$
Thus, \eqref{SLEQ-matrix} is treated as a large-scale deterministic linear complementarity problem:
\begin{equation}\label{RSAA}
0\leq
z
\bot
H^\epsilon z+\bar{q}\geq 0,
\end{equation}
where $z = (x,v_1,...,v_\nu )^T$, $\bar{q}=(a, q_1, \cdots, q_\nu)^T$, and $H^\epsilon$ denotes the
coefficient matrix in \eqref{SLEQ-matrix}.

We have the following assertion of existence and uniqueness of problem \eqref{gs19} by \cite[Theorem 3.1.6]{RPS1992the}.
\begin{proposition}
\label{p5}
For any fixed $\epsilon>0$ and positive integer $\nu$, there exists a unique solution of problem \eqref{gs19}.
\end{proposition}

Recall the result of Lemma \ref{lem3} and the following proposition can be shown in a similar way as in \cite[Proposition 3.7]{CSW2015regularized}.
\begin{proposition}
\label{p6}
Let $(y^{\epsilon}(\xi),\lambda^{\epsilon}(\xi))$ be the unique solution of the regularized second-stage problem (\ref{RM}) for any $(x,\xi)\in\mathbb{R}_+^J\times\Xi$. Then,
\begin{equation*}
\frac{1}{\nu}\sum_{\ell=1}^\nu\lambda^\epsilon(\xi_\ell) \rightarrow \mathbb{E}[\lambda^\epsilon(\xi)]
\end{equation*}
with probability (w.p.) $1$ as $\nu\rightarrow\infty$ uniformly on $\mathbb{B}(x,\delta)\cap \mathbb{R}_+^J$ for any $\delta>0$,
\end{proposition}

Let $x^\epsilon_\nu$ denote the first J-components of the unique solution of problem of \eqref{gs19}, and we have the following assertion.
\begin{lemma}
\label{lem4}
Suppose there exists $p_0>0$ such that for all $j\in\mathcal{J}$, $p_j(\xi)\leq_{a.s.} p_0$. Then, with $\epsilon\downarrow0$, $\{x_\nu^\epsilon\}$ is bounded.
\end{lemma}
We omit the proof since it can be shown analogously as in Proposition \ref{p4}.

\begin{theorem}
\label{Th1}
Suppose there exists $p_0>0$ such that for all $j\in\mathcal{J}$, $p_j(\xi)\leq_{a.s.} p_0$.
Then, for any fixed $\epsilon>0$,
$x^\epsilon_\nu\rightarrow x^\epsilon$ w.p. $1$ as $\nu\rightarrow\infty$.
\end{theorem}
\begin{proof}
From Propositions \ref{p3}, Proposition \ref{p5}, and Lemma \ref{lem4}, for any fixed $\epsilon>0$, both the regularized problem (\ref{RM}) and its SAA-regularized problem (\ref{gs19}) have solutions and contained in some compact subset in $\mathbb{R}_+^J$. We know from Proposition \ref{p6} that
\begin{equation*}
\frac{1}{\nu}\sum_{\ell=1}^\nu\lambda^\epsilon(\xi_\ell) \rightarrow \mathbb{E}[\lambda^\epsilon(\xi)]
\end{equation*}
as $\nu\rightarrow\infty$, uniformly with respect to $x$ on any compact set.
Then, we have $x^\epsilon_\nu\rightarrow x^{\epsilon}$ w.p. $1$ as $\nu\rightarrow\infty$ by \cite[Proposition 19]{S2003monte}.
\qed\end{proof}

Combining Theorem \ref{Th3} with Theorem \ref{Th1},  we have the following convergence result.


\begin{theorem}
\label{Th2}
Suppose there exists $p_0>0$ such that for all $j\in\mathcal{J}$, $p_j(\xi)\leq_{a.s.} p_0$.
Then,
$$\limsup_{\epsilon\downarrow0}\lim_{\nu\to\infty}x^\epsilon_\nu\subseteq S^*$$
w.p. $1$, where $S^*$ denotes the optimal solution set of the first-stage problem of \eqref{SEQ}.
\end{theorem}

\section{Numerical tests and applications on crude oil market}
In this section, we firstly carry out numerical experiments using randomly generated data to illustrate the effectiveness of our model and its solution approach.
Furthermore, we adopt our two-stage stochastic CN equilibrium problem to describe market share competition in crude oil market.
The results show that the model is capable of reproducing the actual oil market share based on in-sample data.
Moreover, it is also shown that the model can make good out-of-sample predictions using historical data.
All the tests are run in MATLAB 2016b on a personal computer with 32GB RAM and 8-core processor ($3.6\times 8 GHz$).

\subsection{Progressive hedging method and smoothing Newton sub-algorithm}
Firstly, randomly generated problems are used for testing our regularized SAA approach to solve the two-stage stochastic CN equilibrium problem.
Recall that the model of interests takes the form of a scenario-based linear complementarity problem (\ref{gs19}) or its equivalent expression (\ref{SLEQ-matrix}) with sufficiently small $\epsilon$.
The solution process adopts the well-known PHM.
The PHM is globally convergent and the convergence rate is linear for problem (\ref{SLEQ-matrix}) with SAA approach.\\
\newpage
\newpage

The classical PHM to solve \eqref{SLEQ-matrix} is as follows

\bigskip
\noindent
\fbox{%

  \parbox{\textwidth}{%
{\bf  Algorithm 1:  Progressive hedging method}

{\bf Step 0.} Given an initial point $x^0 \in \mathbb{R}^J$, let $x_\ell^0=x^0\in \mathbb{R}^J, v_\ell^0 \in \mathbb{R}^{2J}$ and $w_\ell^0 \in \mathbb{R}^J$, for $\ell=1,\ldots,\nu$,   such that $\frac{1}{\nu}\Sigma_{\ell=1}^\nu w_\ell^0=0.$ Set the initial point $z^0=(x^0, v_1^0,\ldots,v_\nu^0)^T$. Choose a step size $r>0.$  Set $k=0.$

{\bf Step 1.} If the point $z^k$  satisfies the  condition
$$ \|\min(z^k, H^\epsilon z^k+\bar{q}) \|\leq 10^{-6}, $$
output the solution $z^k$ and terminate the algorithm; otherwise,  go to Step 2.

{\bf Step 2.}  For $\ell=1,\ldots,\nu$, find $(\hat{x}^k_\ell, \hat{v}_\ell^k)$ that solves linear  complementarity problems
\begin{equation}\label{PHM-subproblem}
  \begin{array}{l}
  0\leq x_\ell \bot C x_\ell +B v_\ell+a +w_\ell^{k} +r(x_\ell-x_\ell^k)  \geq 0,  \\
  0\leq v_\ell \bot  -B^T x_\ell+D^\epsilon_\ell v_\ell  +q_{\ell}+r(v_\ell-v_\ell^k)  \geq 0.
\end{array}
\end{equation}
Then let $\bar{x}^{k+1}=\frac{1}{\nu} \sum_{\ell=1}^\nu \hat{x}_\ell^k$,  and for $\ell=1,\ldots,\nu,$ update
\begin{equation}
  x_\ell^{k+1}=\bar{x}^{k+1}, v_\ell^{k+1}=\hat{v}_\ell^{k}, w_\ell^{k+1}=w_\ell^{k}+r(\hat{x}_\ell^{k}-x_\ell^{k+1}),\notag
\end{equation}
to get point $z^{k+1}=(\bar{x}^{k+1},v_1^{k+1},\ldots,v_\nu^{k+1})^T$.

{\bf Step 3.} Set $k:=k+1;$  go back to Step 1.\vspace{2mm}
  }%
}
\bigskip

The PHM   involves solving $\nu$ independently sample-based LCP \eqref{PHM-subproblem} at each iteration{\color{blue}.} Problem (\ref{PHM-subproblem}) is  well-defined, since for  $\ell=1,\ldots,\nu,$ the coefficient matrix
$$\left(
  \begin{array}{cc}
    C+rI & B \\
    -B^T & D^\epsilon_\ell+rI \\
  \end{array}
\right)\in \mathbb{R}^{3J\times 3J}$$
is positive definite for any $\epsilon>0$.
Thus, it has a unique solution for each $\ell=1,\ldots,\nu$.
For simplicity, denote  (\ref{PHM-subproblem}) as
\begin{equation}\label{PHM-LCP}
0\leq z_\ell\bot \tilde{H}^\epsilon_\ell z_\ell+\tilde{q}_\ell\geq 0,
\end{equation}
where $ z_\ell=\left(
\begin{array}{c}
x_\ell \\
v_\ell \\
\end{array}
\right),
\tilde{H}_\ell^\epsilon=\left(
  \begin{array}{cc}
    C+rI & B \\
    -B^T & D^\epsilon_\ell+rI \\
  \end{array}
\right),
 \tilde{q}_\ell=\left(\begin{array}{cc}a +w_\ell^{k} -rx_\ell^k\\q_{\ell}-r v_\ell^k\end{array}\right)$ with $\ell=1,2,\ldots,\nu.$
 Then, we can equivalently write a large-scale LCP for all $\nu$ samples and solve
  \begin{equation}\label{PHM-s-whole}
  0\leq\left(
    \begin{array}{c}
      z_1 \\
      z_2 \\
      \vdots \\
      z_\nu\ \\
    \end{array}
  \right)
  \bot
  \left(
    \begin{array}{cccc}
      \tilde{H}_1^\epsilon &  &  &  \\
       & \tilde{H}_2^\epsilon &  &  \\
       &  & \ddots &  \\
       &  &  & \tilde{H}_\nu^\epsilon \\
    \end{array}
  \right)
  \left(
    \begin{array}{c}
      z_1 \\
      z_2 \\
      \vdots \\
      z_\nu \\
    \end{array}
  \right)
  +
  \left(
    \begin{array}{c}
      \tilde{q}_1 \\
       \tilde{q}_2 \\
      \vdots \\
       \tilde{q}_\nu \\
    \end{array}
  \right)\geq 0.
\end{equation}
The structure of problem (\ref{PHM-s-whole}) enables us to use  block computation to solve it, which can significantly improves the efficiency of the PHM. For example, suppose $\nu=mN$ with $m, N$ being positive integer number.  The equivalent $m$-block reformulation for  \eqref{PHM-s-whole} reads

    \begin{equation}\label{PHM-BLOCK}
  0\leq\left(
    \begin{array}{c}
      z_{i_1} \\
      z_{i_2} \\
      \vdots \\
      z_{i_m} \\
    \end{array}
  \right)
  \bot
  \left(
    \begin{array}{cccc}
      \tilde{H}_{i_1}^\epsilon &  &  &  \\
       & \tilde{H}_{i_2}^\epsilon &  &  \\
       &  & \ddots &  \\
       &  &  & \tilde{H}_{i_m}^\epsilon \\
    \end{array}
  \right)
  \left(
    \begin{array}{c}
      z_{i_1} \\
      z_{i_2} \\
      \vdots \\
      z_{i_m} \\
    \end{array}
  \right)
  +
  \left(
    \begin{array}{c}
      \tilde{q}_{i_1} \\
       \tilde{q}_{i_2} \\
      \vdots \\
       \tilde{q}_{i_m} \\
    \end{array}
  \right)\geq 0,
\end{equation}
where $i_j=iN+j,  j=1,\ldots, m,  i=0, \ldots, (m-1)N$.

The main computation cost of the PHM is in {\bf Step 2} due to the large sample size $\nu$ despite the relative low cost in solving for one sample.
Block implementation speeds up the computation by better exploring and rebalancing the computational load.
In practice, the number of blocks are adjusted so that the over all computation time is minimized.
To improve the efficiency of the PHM, we also use the warm-start technique suggested in \cite{RS2018solving} to choose an initial point for  subproblem  (\ref{PHM-subproblem}).
More specifically, the solution $z^k$ of subproblem \eqref{PHM-subproblem} at the $k$-th iteration is used as a starting point for the $(k+1)$-th iteration.

In the remaining of this subsection, we focus on problem \eqref{PHM-LCP} for a given sample and omit the subscription $_\ell$ for ease of expression.
In order to take advantage of the sparse structure of the subproblem \eqref{PHM-LCP}, we apply the smoothing Newton method proposed by Chen and Ye \cite{CY1999homotopy} to solve it. In what follows, we give a brief description of the smoothing Newton method.
It is well-known that solving (\ref{PHM-LCP}) for a given sample is equivalent to solving the nonsmooth equation
\begin{equation}\label{non-smooth-problem}
F(z)=\min(z, \tilde{H}^\epsilon z+\tilde{q} )=0.
\end{equation}
The main idea of the smoothing Newton method is to use a smooth approximation function to approximate the nonsmooth function $F$ and then solve the corresponding linear system.
We use the smooth  Gariel-Mor{\'e} approximation function $f: \mathbb{R}^{3J}\times \mathbb{R}_{++}\rightarrow \mathbb{R}^{3J}$  to approximate the nonsmooth function $F$. The density function is $\rho(s)=2/(s^2+4)^{\frac{3}{2}}$ (see \cite{CY1999homotopy} for details). 
In our  numerical tests,  the $j$-th component of  $f$ reads
$$  f_j(z,\delta)=(z)_j-\frac{1}{2}\left(\sqrt{(\tilde{H}^\epsilon z +\tilde{q}-z)_j^2+4\delta^2}+(z-\tilde{H}^\epsilon z-\tilde{q})_j\right),  \quad j=1,\ldots,3J, $$
 where the corresponding $j$-th diagonal element of the Jacobian $\bar{D}(z)$ is
$$ \bar{D}_{jj}(z)=\frac{1}{2}\left(\frac{(z-\tilde{H}^\epsilon z-\tilde{q})_j}{\sqrt{(z-\tilde{H}^\epsilon z-\tilde{q})_j^2+4\delta^2}}+1 \right) , \quad j=1,\ldots,3J. $$
Then, the smoothing Newton method for solving subproblem (\ref{non-smooth-problem}) requires to solve a linear equation to determine  $d^k$ at each iteration, namely
\begin{equation}\label{sn-subproblem}
    \nabla_{z}f(z^k,\delta_k) d^k+F(z^k)=0,
\end{equation}
where $\nabla_{z}f(z^k,\delta_k)=I-\bar{D}(z^k)(I-\tilde{H}^\epsilon)$, $\delta_k$ decreases to 0 according to the criterion in article \cite{CY1999homotopy}. To guarantee the well-defineness of  \eqref{sn-subproblem}, we make use of the following result.
\begin{theorem}\cite{GM1997smoothing}\label{th:nonsingularity}
For any diagonal matrix $\tilde{D}=diag(\tilde{D}_{jj})\in \mathbb{R}^{J\times J}$ with $0\leq \tilde{D}_{jj}\leq 1, j=1,2,\ldots, J$,  the matrix $I-\tilde{D}(I-A)$  is nonsingular if and only if $A$ is a P-matrix.
\qed\end{theorem}
For any given sample, it is known that $H^\epsilon$ is positive definite and hence a P-matrix.
Moreover,  $\bar{D}(z)$ is a diagonal matrix with its element on the interval $[0,1]$ for any $(z,\delta)\in \mathbb{R}^{3J}\times \mathbb{R}_{++}$.
Therefore, using Theorem \ref{th:nonsingularity}, the Jacobian $\nabla_{z}f(z,\delta)$ is nonsingular for any $(z,\delta)\in \mathbb{R}^{3J}\times \mathbb{R}_{++}$. Thus, the linear equation \eqref{sn-subproblem} is well-defined.

Denoting the matrix
 $\bar{D}(z)=diag(\bar{D}_1(z),\bar{D}_2(z),\bar{D}_3(z))$, the Jacobian  $\nabla _{z} f(z,\delta)$ at the point $z$ is of the following  structure
\begin{eqnarray}
\nabla_{z} f(z,\delta)&= &  (I- \bar{D}(z))+\bar{D}(z)\tilde{H}^\epsilon)  \notag \\
         & \triangleq&   \left(
                  \begin{array}{ccc}
                    \Lambda_1(z) & 0 &  -\bar{D}_1(z) \\
                    0 & u_1(z)e^T +\Lambda_2(z) & \bar{D}_2(z) \\
                    \bar{D}_3(z) & -\bar{D}_3(z)  & \Lambda_3(z) \\
                  \end{array}   \label{smoothing-matrix}
                \right),
\end{eqnarray}
where $\Lambda_1(z)= \bar{D}_1(z)(C+(r-1)I)+I$, $\Lambda_2(z)=(\gamma(\xi)+(r-1))\bar{D}_2(z)+I$,
 $\Lambda_3(z)= (\epsilon+(r-1))\bar{D}_3(z)+I$ are all diagonal matrices, and $u_1(z)=\gamma(\xi)\bar{D}_2(z)e$.
We can take advantage of the sparse structure of (\ref{smoothing-matrix}) in solving \eqref{sn-subproblem}.
Specifically, $\nabla_{z}f(z,\delta)$ consists of only diagonal sub-matrix and the matrix $u_1(z)e^T +\Lambda_2(z)$, where the later is a sum of a diagonal sub-matrix and a rank-one matrix.

Noticing from  \eqref{smoothing-matrix}, linear equation  \eqref{sn-subproblem}  is of the following form
\begin{equation}\label{smooth-linear-eq}
\left(
  \begin{array}{ccc}
    \Lambda_1 & 0 & \Lambda_2 \\
    0 & u_1u_2^T+\Lambda_3 & \Lambda_4 \\
    \Lambda_5 & \Lambda_6 & \Lambda_7 \\
  \end{array}
\right)
\left(
  \begin{array}{c}
    s_1 \\
    s_2 \\
    s_3 \\
  \end{array}
\right)
=\left(
   \begin{array}{c}
     b_1 \\
     b_2 \\
     b_3 \\
   \end{array}
 \right),
\end{equation}
where $\Lambda_i\in \mathbb{R}^J , i=1,2, \ldots, 7$ are diagonal matrices with $\Lambda_1$ and $\Lambda_3$ being nonsingular, $s_i\in \mathbb{R}^J, b_i\in \mathbb{R}^J, i=1,2,3,$ and  $u_1, u_2 \in \mathbb{R}^J.$
For ease of notation, we use $\hat{\Lambda}=diag(1/\Lambda_{11}, \ldots, 1/\Lambda_{JJ})$ to represent the inverse of any invertible diagonal matrix $\Lambda=diag(\Lambda_{11}, \ldots, \Lambda_{JJ})$.
Given the sparse structure of the coefficient matrix, we can solve  \eqref{smooth-linear-eq} efficiently. More exactly, by the first two equations of  \eqref{smooth-linear-eq}, we can get
\begin{align}\label{s1}
  &s_1=\hat{\Lambda}_1(b_1-\Lambda_2 s_3 ),\\
  &s_2=(u_1 u_2^T+\Lambda_3)^{-1}(b_2-\Lambda_4 s_3).\label{s2}
\end{align}
Directly substituting  \eqref{s1} and \eqref{s2} into the third equation of \eqref{smooth-linear-eq}, then we have
\begin{equation}\label{s3}
  (\Lambda_7-\Lambda_5\hat{\Lambda}_1\Lambda_2-\Lambda_6(u_1u_2^T+\Lambda_3)^{-1}\Lambda_4)s_3=b_3+const,
\end{equation}
where $const=-(\Lambda_5\hat{\Lambda}_1 b_1+\Lambda_6(u_1u_2^T+\Lambda_3)^{-1}b_2)$.
For computing the inverse matrix of  $(u_1u_2^T+\Lambda_3)$, the Sherman-Morrison formula is useful.
We have by the Sherman-Morrison formula
\begin{equation}\label{sm-inverse}
  (u_1u_2^T+\Lambda_3)^{-1}=\hat{\Lambda}_3-\frac{\hat{\Lambda}_3 u_1 u_2^T\hat{\Lambda}_3}{1+u_2^T\hat{\Lambda}_3u_1}.
\end{equation}
Substituting  \eqref{sm-inverse} into \eqref{s3}, we get
\begin{equation}\notag
  (\Lambda_0+\alpha \tilde{u}_1 \tilde{u}_2^T)s_3=b_3+const,
\end{equation}
where $\alpha=1/(1+u_2^T\hat{\Lambda}_3u_1),
\Lambda_0=\Lambda_7-\Lambda_5\hat{\Lambda}_1\Lambda_2-\Lambda_6\hat{\Lambda}_3\Lambda_4, \tilde{u}_1=\Lambda_6\hat{\Lambda}_3 u_1, \tilde{u}_2=\Lambda_4\hat{\Lambda}_3 u_2.$
Then, if $\Lambda_0$ is nonsingular, using the Sherman-Morrison formula again, we can immediately get the solution of $s_3$
\begin{equation}\label{s3-solution}
  s_3=\left(\hat{\Lambda}_0-\frac{\alpha\hat{\Lambda}_0 \tilde{u}_1 \tilde{u}_2^T\hat{\Lambda}_0}{1+\alpha \tilde{u}_2^T\hat{\Lambda}_0 \tilde{u}_1}\right)(b_3+const).
\end{equation}
Then, substituting the $s_3$ into \eqref{s1} and \eqref{s2}, we  get the solution of $s_1$ and $s_2$, respectively.

From \eqref{s3-solution}, one can know that the computation cost of $s_3$ is  trivial, since we only need to compute inverse of several diagonal matrix, namely, $\Lambda_0, \Lambda_1,$ and $\Lambda_3$. Once $s_3$ is obtained, the calculation of $s_1$ and $s_2$ just needs to perform matrix-vector production.
Therefore, the linear equation $\eqref{smooth-linear-eq}$ can be solve efficiently.

\subsection{Randomly generated problems}
For the first part of numerical test, we randomly generated the problem of the form (\ref{SLEQ-matrix}).
More specifically, we generate a set of i.i.d.  samples $\{\xi_\ell\}_{\ell=1}^\nu$ from a uniformly distribution over the interval $[0,1]$.
For  $\ell=1,\ldots, \nu$, set
\begin{equation*}
  \begin{split}
      p(\xi_\ell) =((\xi_\ell)_1,(\xi_\ell)_2, \ldots,(\xi_\ell)_J )^T, \quad \quad
      \Pi (\xi_\ell)=\gamma(\xi_\ell)(ee^T+I)\triangleq(\xi_\ell)_1(ee^T+I).
  \end{split}
\end{equation*}
Diagonal matrix $C\in \mathbb{R}^{J\times J}$ and $a\in \mathbb{R}^J$ are generated with its elements uniformly distributed over the interval $[1,2]$.
All the numerical results are based on the average of 10 independent runs.

To show the feasibility of the solution of the regularized problem compared to that of  the original problem,  we compute the following residual value
\begin{equation*}
     {\rm{\bf Res}}=\|\min(z, Hz+\bar{q})\|,
\end{equation*}
where $H$ denotes the coefficient matrix  of \eqref{SLEQ-matrix} with $\epsilon=0$.

Selected numerical results  for $J=10$ were listed in the  Table  \ref{table:player-10} and \ref{table:player-10-block}. The first table is the solution obtained by PHM without block computation, while table \ref{table:player-10-block} illustrates results with block implementation.
For our generated experiments, we found $N=50$ is the best block choice (measured by CPU time) for $J=10$.
The average number of iterations, the average cpu time,  and the average value of  {\bf Res} were recorded in both table.
It is easy to see that the block implementation greatly reduces the cpu time.
For the same value of $\epsilon$, the number of iterations increases slightly when the sample size $\nu$ increases.
In cases where the sample size $\nu$ is kept constant and the values of regularization parameter $\epsilon$ are chosen from $\epsilon=10^{-3}$ to $\epsilon=10^{-12}$, the iteration numbers are barely influenced as well as the cpu time.
Furthermore, we observe the convergence of our regularization approach with decreasing values of $\epsilon$, as have been proved in previous sections.
Also notice that,  the value of  {\bf Res} decreases when the $\epsilon$ diminishes from $10^{-3}$ to $10^{-12}$.
Numerically, it shows that the solution of the regularized problem is also that of the original problem when $\epsilon=10^{-12}$.

Figure \ref{figure:player-10-iter-time} illustrates the performance of the PHM measured by the number of iterations and cpu time for $10$ players.
It is also worth mentioning that although one might expect the problem to be more difficult to solve for a small $\epsilon$, the numerical performance in our experiments remain roughly unaffected with decreasing values of $\epsilon$.

Figure \ref{figure:player-10} demonstrates the convergence property of the first-stage solution $x$ when the sample size gets large for the case $J=10$.
The convergence trend can be seen component-wisely as the solution $x$ converges when the sample size $\nu$ gets large.

\begin{table}[h]\scriptsize
\begin{center}

\begin{tabular}{cc|ccc|ccc}
\hline
 $\nu$ &  $J(1+2\nu)$ & Iter  & CPU time/s     & ${\bf Res}$           & Iter  & CPU time/s     & ${\bf Res}$ \\
\hline  \multicolumn{2}{c|}{}   & \multicolumn{3}{c|}{$\epsilon=10^{-3}$} & \multicolumn{3}{c}{$\epsilon=10^{-6}$}\\
\hline
                     10 &       210             &     146.30 &       0.26   &   4.42e-01         & 176.20 &       0.32   &   3.88e-04 \\
            50 &      1010           &     194.70 &       1.81   &   9.35e-01         & 197.40 &       1.83   &   9.32e-04 \\
            500 &     10010        &     208.70 &     26.72   &   3.00e+00        & 212.20 &      27.21  &   2.99e-03 \\
            2000 &     40010      &     222.60 &     154.97 &   5.93e+00        & 220.50 &     153.54 &   6.00e-03 \\
            5000 &    100010     &     224.70 &     623.53 &   9.49e+00        & 226.40 &     627.53 &   9.48e-03 \\
\hline
\multicolumn{2}{c|}{}   & \multicolumn{3}{c|}{$\epsilon=10^{-9}$} & \multicolumn{3}{c}{$\epsilon=10^{-12}$}\\
\hline
           10 &       210 &     152.70 &       0.27   &   1.08e-06  &     169.40 &       0.30   &   9.49e-07  \\
           50 &     1010 &    197.20  &       1.83   &   1.41e-06  &     194.40 &       1.80   &   9.75e-07  \\
         500 &     10010 &     212.70 &    27.21  &   3.21e-06  &     209.70 &     26.85   &   9.59e-07  \\
       2000 &     40010 &     220.30 &   153.34 &   6.16e-06  &     220.70 &     153.73 &   9.51e-07  \\
      5000 &    100010 &     226.70 &   628.89 &   9.60e-06  &     226.20 &     627.58   &   9.60e-07  \\
\hline
\end{tabular}
\end{center}
\caption{Numerical results for different  $\epsilon$ and sample size $\nu$, $J=10$ with individual sample. }
\label{table:player-10}
\end{table}

\begin{table}[h]\scriptsize
\begin{center}

\begin{tabular}{cc|ccc|ccc}
\hline
           $\nu$ &  $J(1+2\nu)$ & Iter  & CPU time/s     & ${\bf Res}$           & Iter  & CPU time/s     & ${\bf Res}$ \\
\hline  \multicolumn{2}{c|}{}   & \multicolumn{3}{c|}{$\epsilon=10^{-3}$} & \multicolumn{3}{c}{$\epsilon=10^{-6}$}\\
\hline
           10 &       210             &     162.60 &       0.21   &   4.04e-01         & 155.50 &       0.20   &   4.14e-04 \\
            50 &      1010           &     180.60 &       0.94   &   9.42e-01         & 184.00 &       0.93   &   9.31e-04 \\
            500 &     10010        &     203.30 &      10.67   &   2.97e+00        & 204.10 &      10.61   &   2.98e-03 \\
            2000 &     40010      &     213.80 &      42.38   &   5.89e+00        & 213.70 &      42.36   &   5.91e-03 \\
            5000 &    100010     &     218.20 &     115.15  &   9.36e+00        & 219.60 &     115.63   &   9.36e-03 \\
\hline
\multicolumn{2}{c|}{}   & \multicolumn{3}{c|}{$\epsilon=10^{-9}$} & \multicolumn{3}{c}{$\epsilon=10^{-12}$}\\
\hline
           10 &       210 &     143.20 &       0.19   &   1.08e-06   &     149.50 &       0.19   &   9.59e-07 \\
           50 &     1010 &    197.40 &       1.02   &   1.38e-06   &     191.10 &       0.97   &   9.58e-07  \\
         500 &     10010 &     205.10 &      10.58   &   3.19e-06 &     202.40 &      10.67   &   9.52e-07  \\
       2000 &     40010 &    214.10 &      42.46   &   6.07e-06  &      213.40 &      42.62   &   9.73e-07   \\
      5000 &    100010 &     219.00 &     115.35   &   9.50e-06 &     219.70 &     115.58   &   9.74e-07  \\
\hline
\end{tabular}
\end{center}
\caption{Numerical results for different  $\epsilon$ and sample size $\nu$, $J=10$ with block implementation. }
\label{table:player-10-block}
\end{table}

\begin{figure}[!htbp]
  \centering
  { \includegraphics[height=4.8cm]{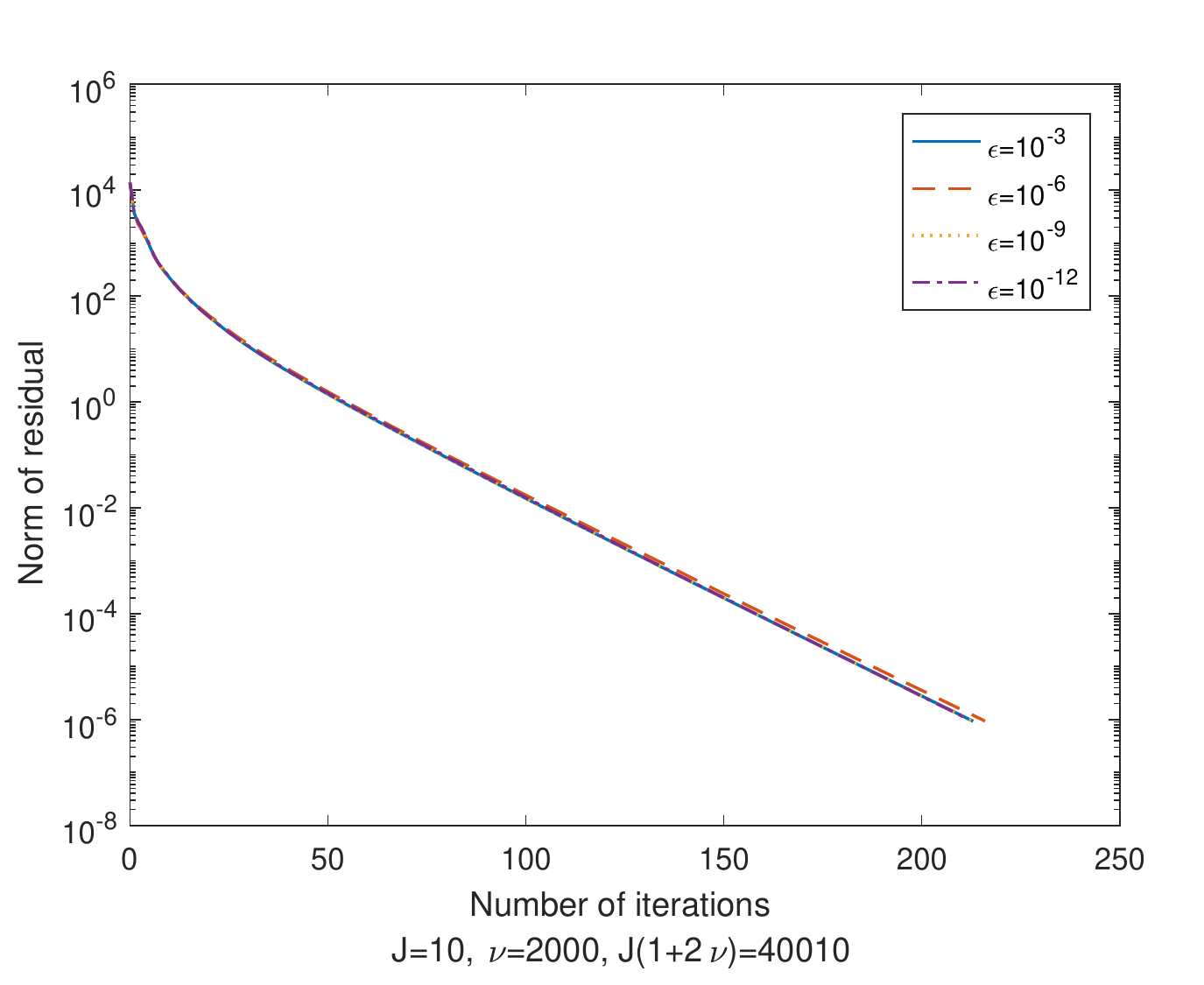}}  \hspace{0em}%
  { \includegraphics[height=4.8cm]{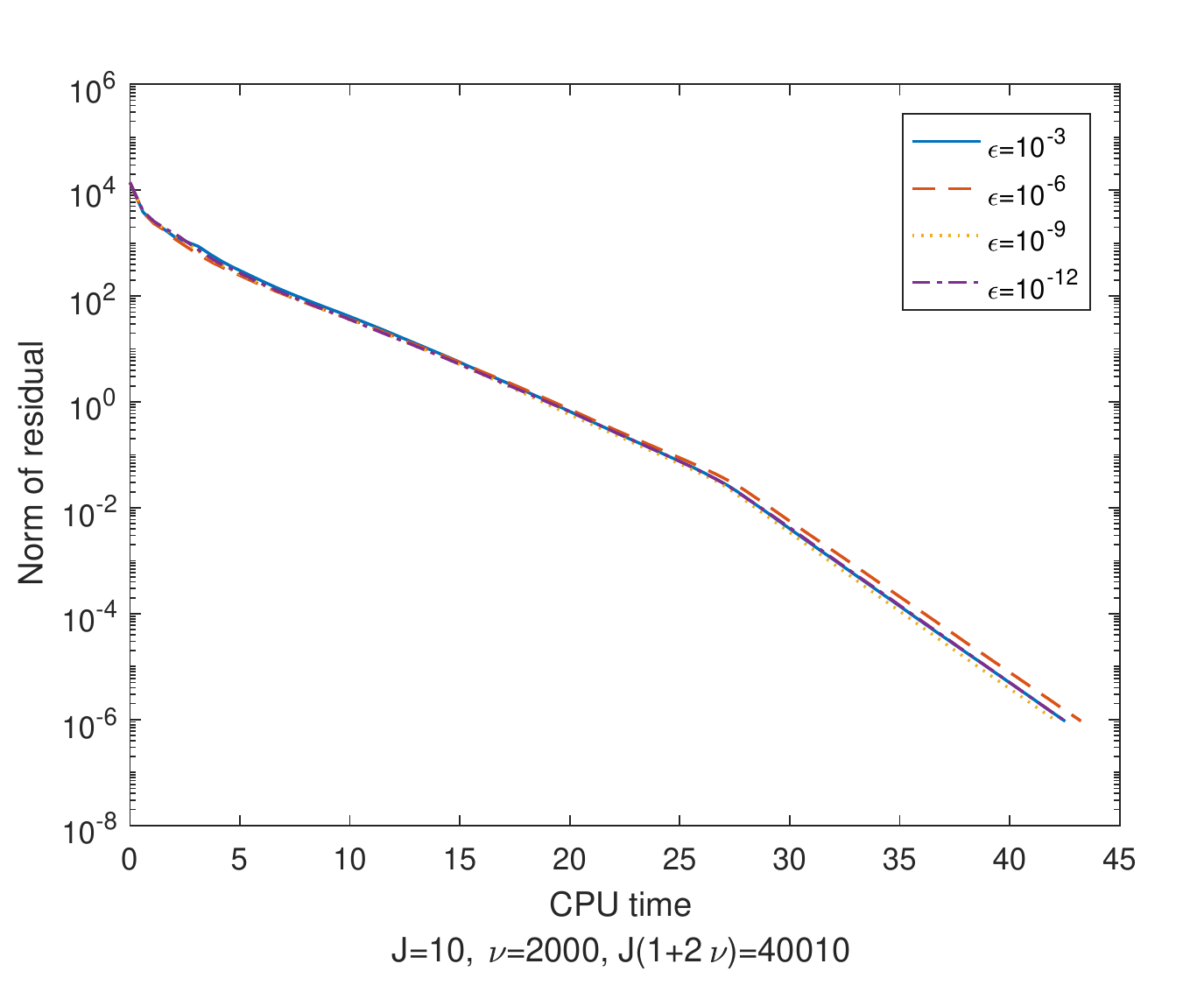}}

  {\includegraphics[height=4.8cm]{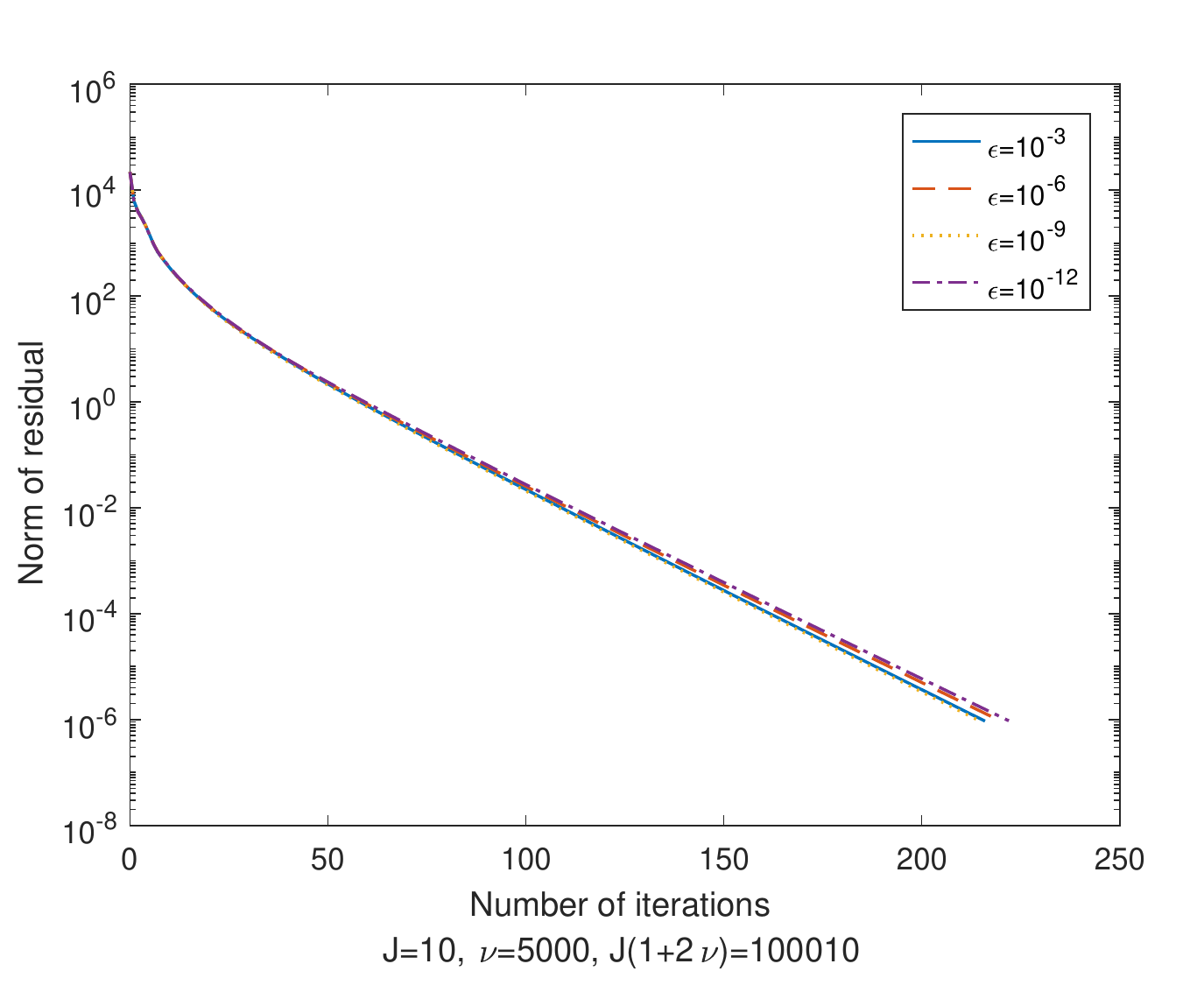}}  \hspace{0em}%
  {\includegraphics[height=4.8cm]{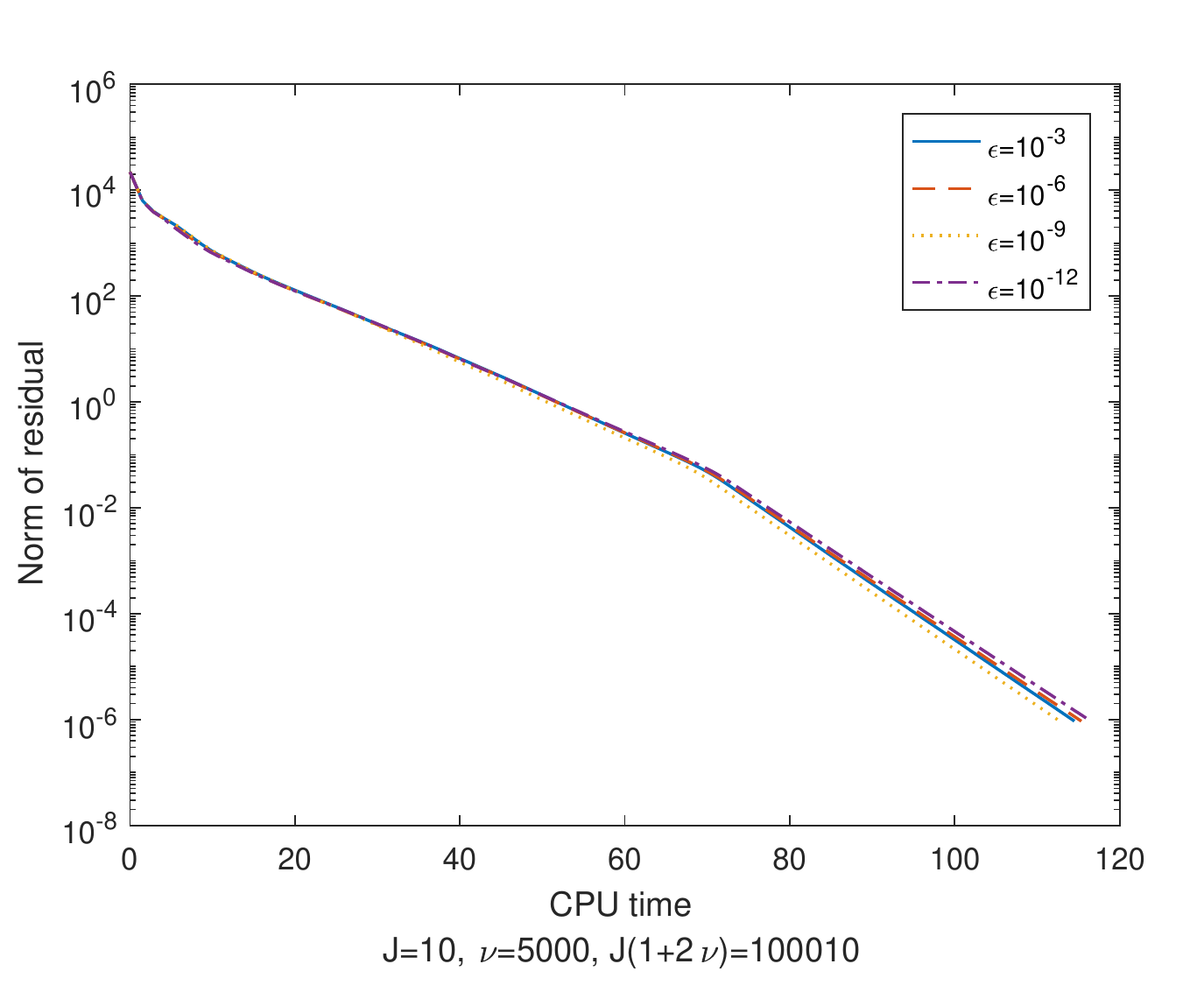}}
  \caption{Numerical comparisons among different $\epsilon$, $J=10$.}
  \label{figure:player-10-iter-time}
\end{figure}

%

\begin{figure}[!htbp]
  \centering

   { \includegraphics[height=4.7cm]{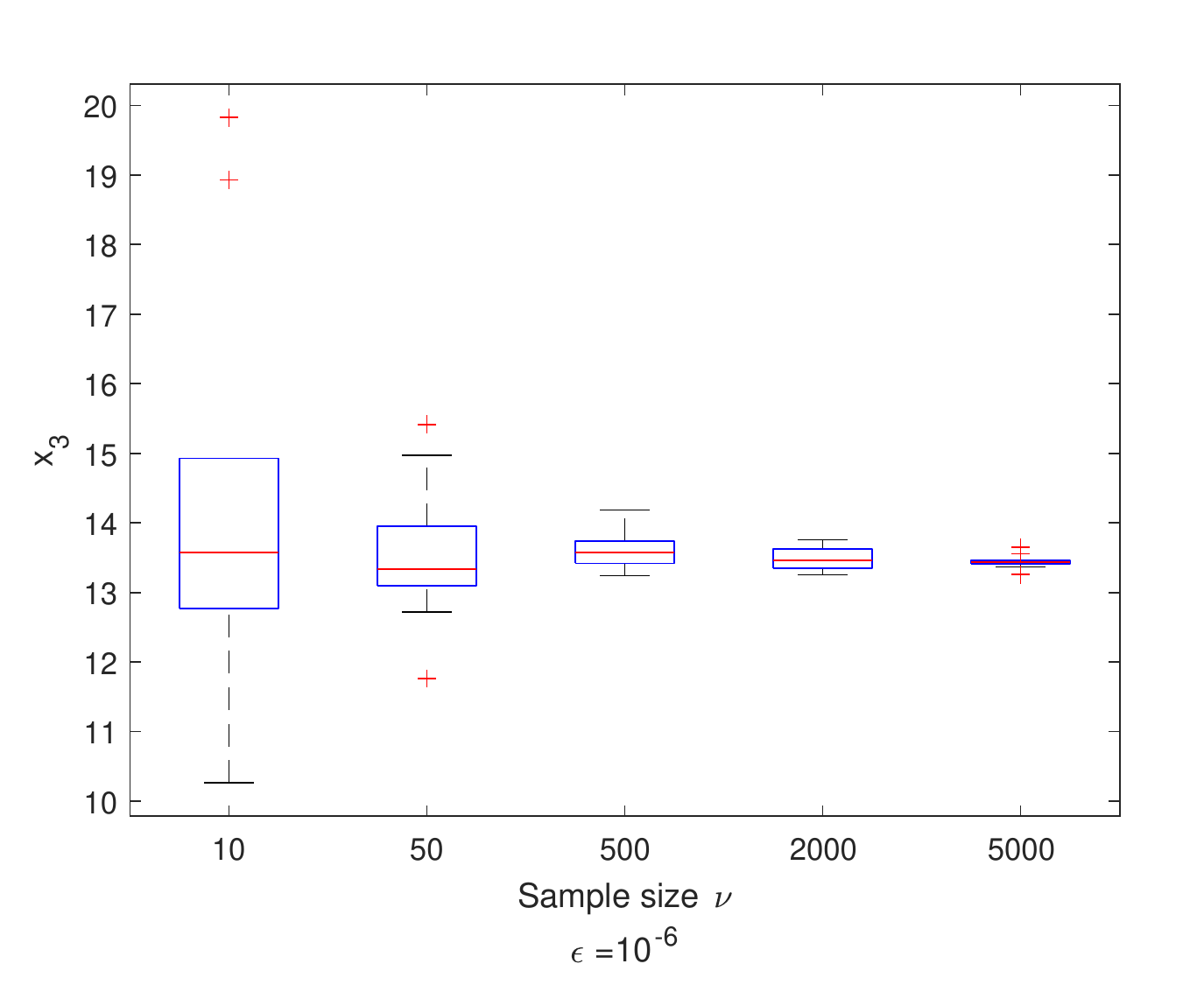}}
   { \includegraphics[height=4.7cm]{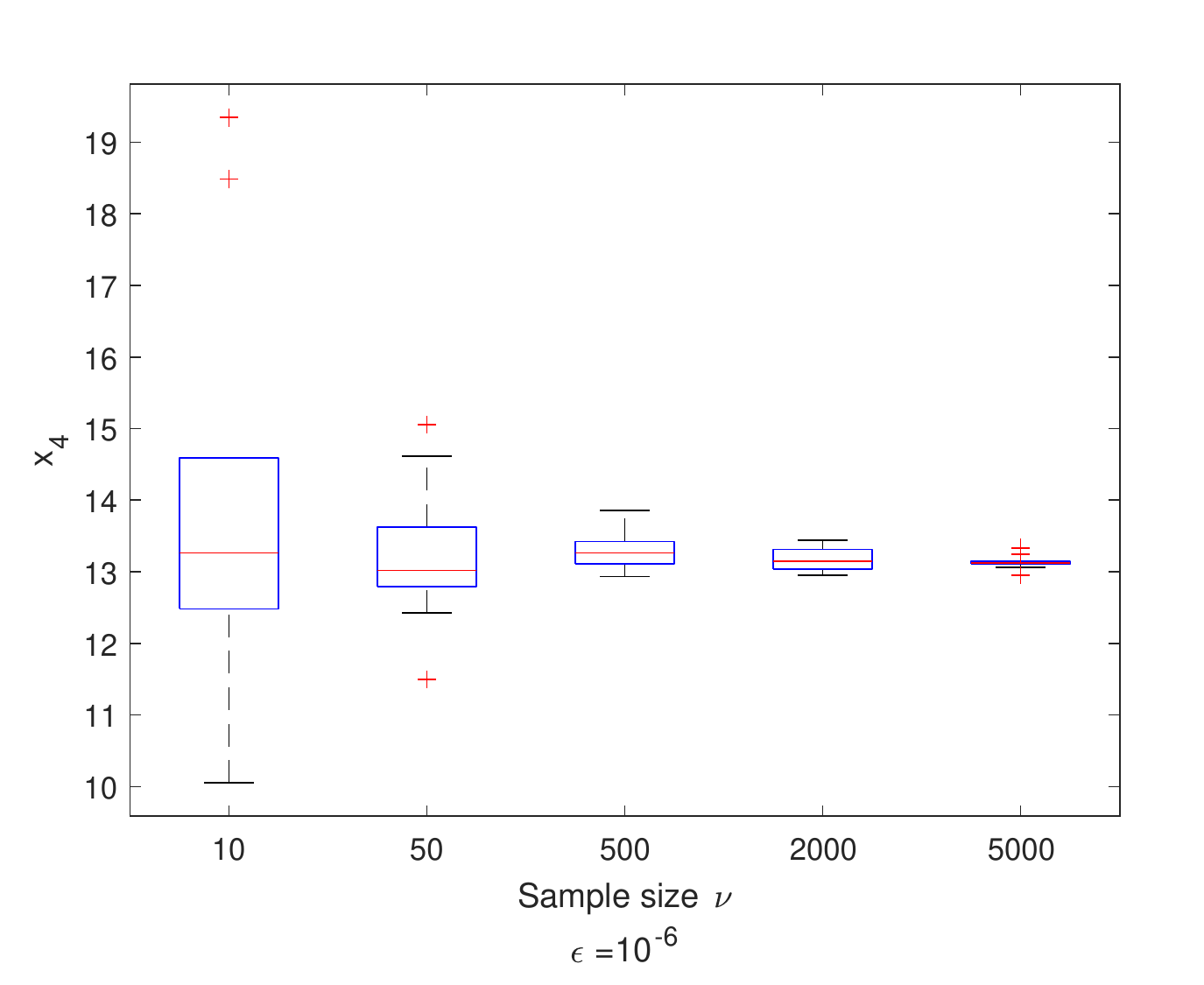}}

   { \includegraphics[height=4.7cm]{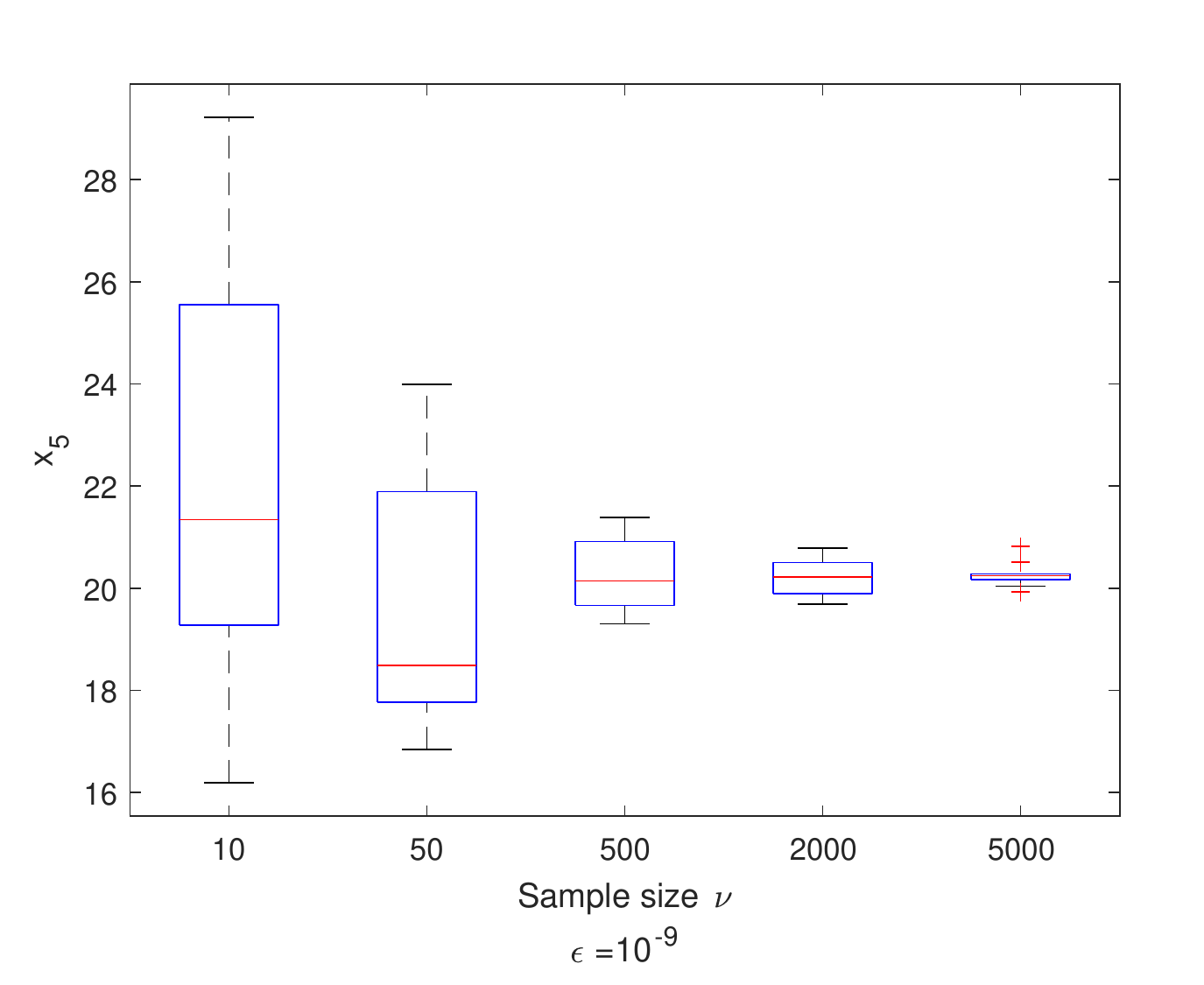}}
   { \includegraphics[height=4.7cm]{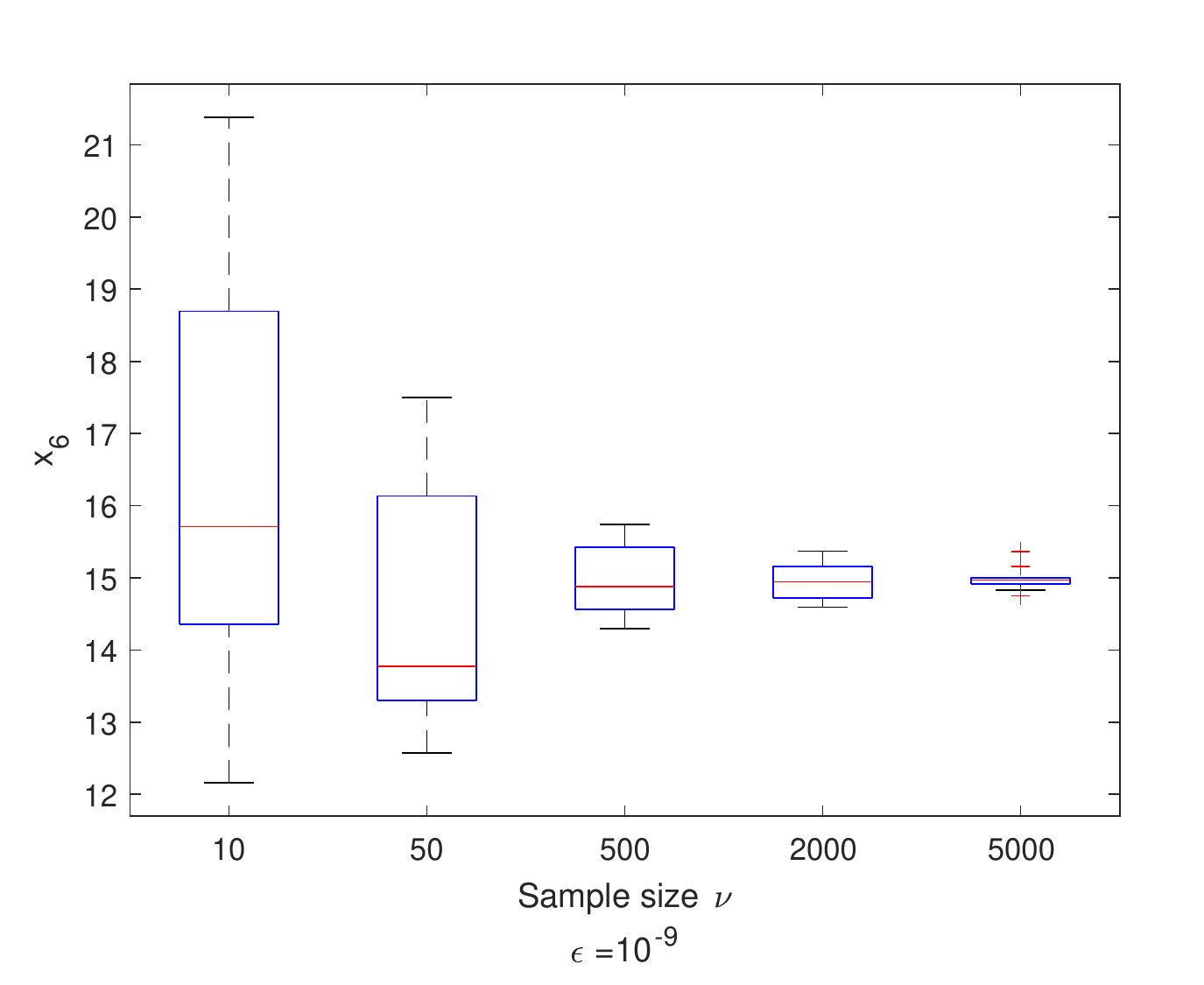}}

   { \includegraphics[height=4.7cm]{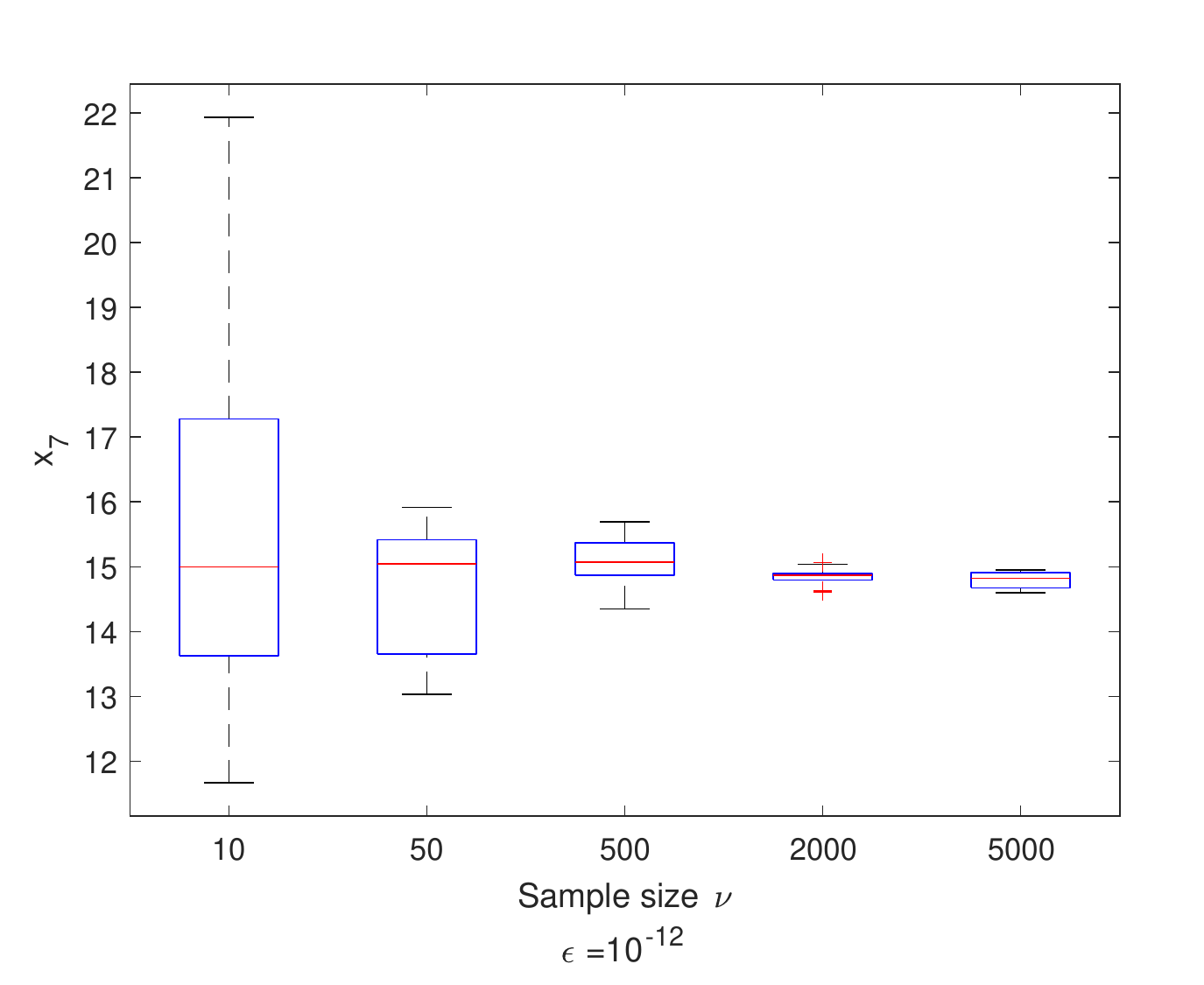}}
   {\includegraphics[height=4.7cm]{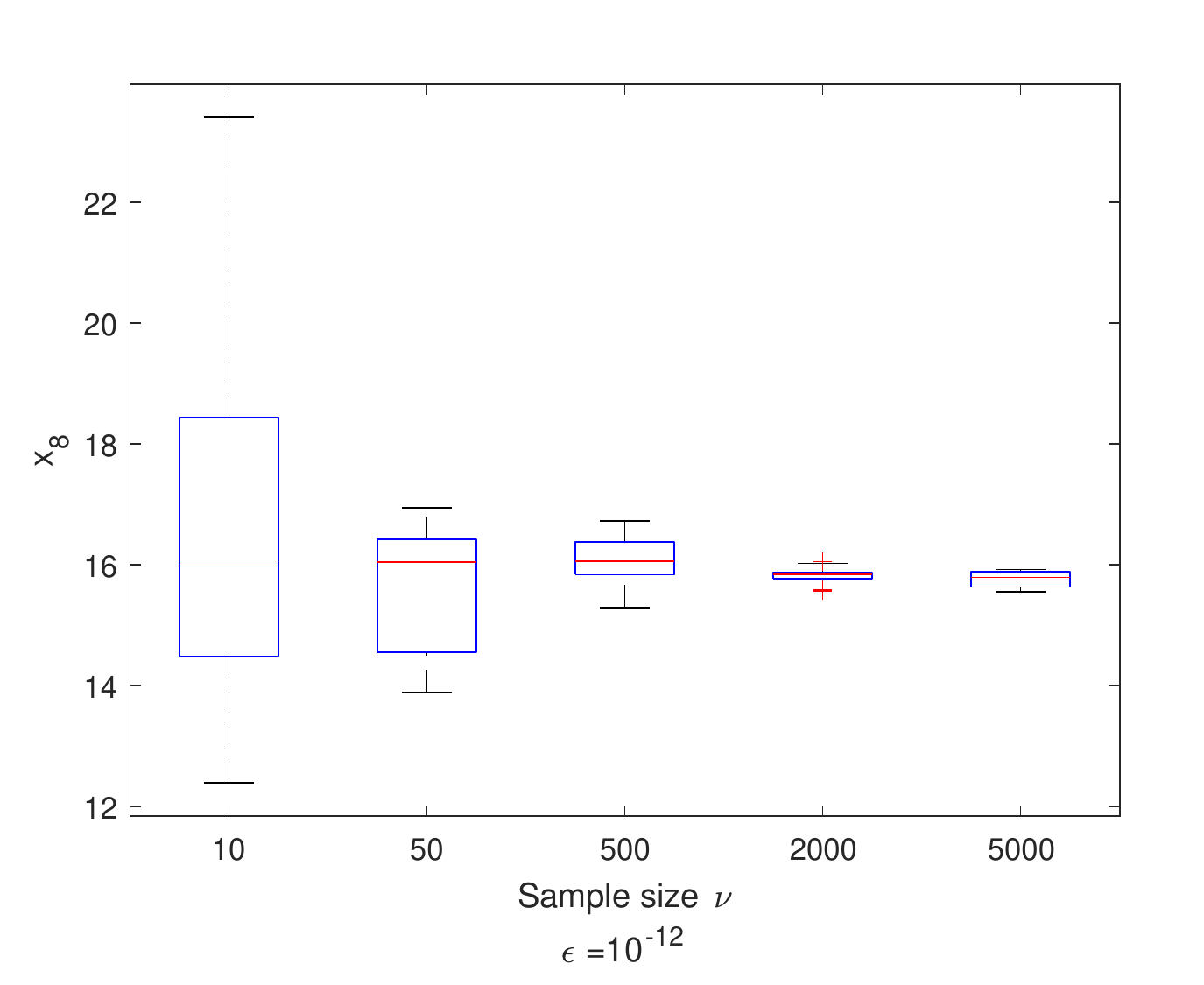}}

   { \includegraphics[height=4.7cm]{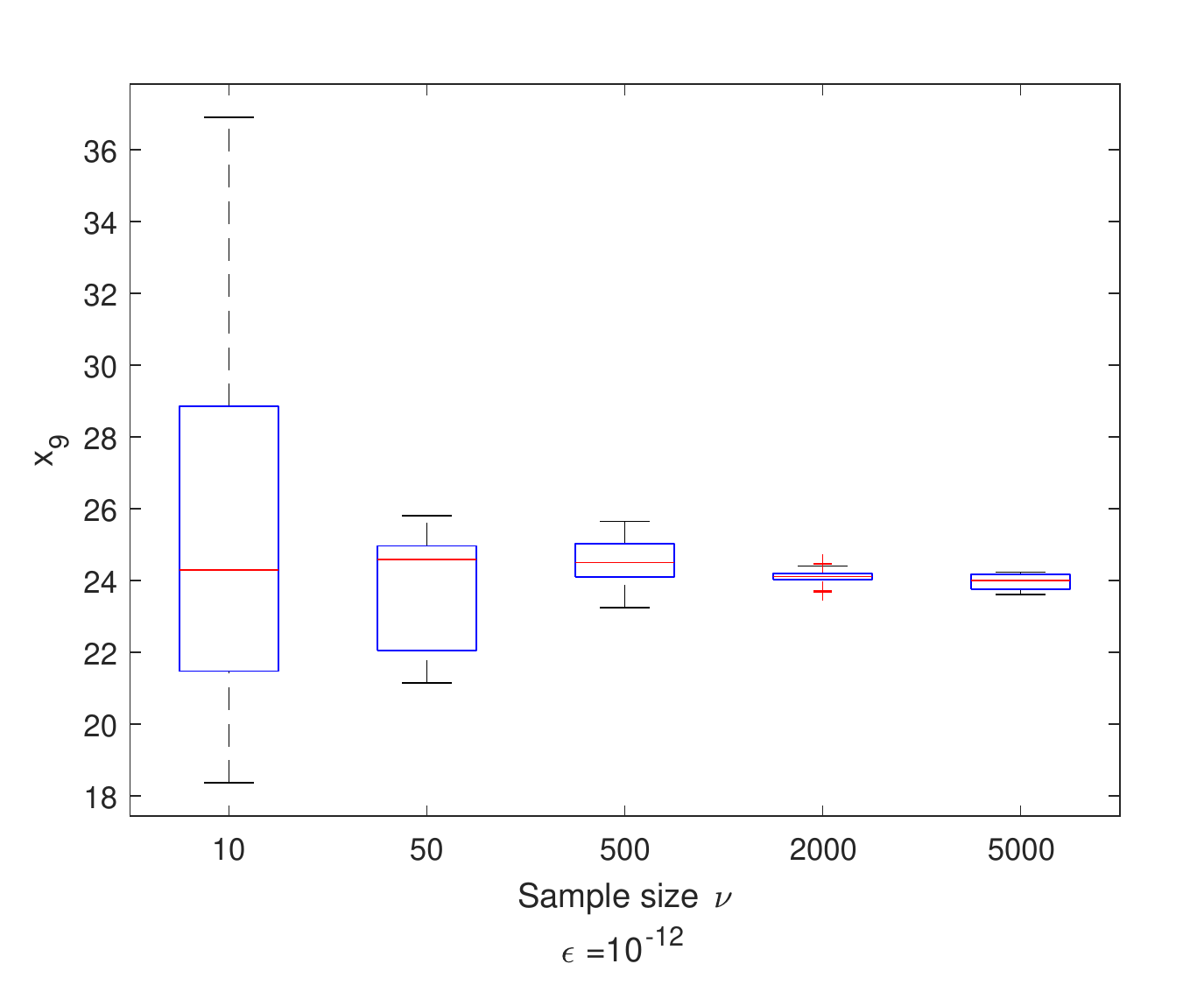}}
   {\includegraphics[height=4.7cm]{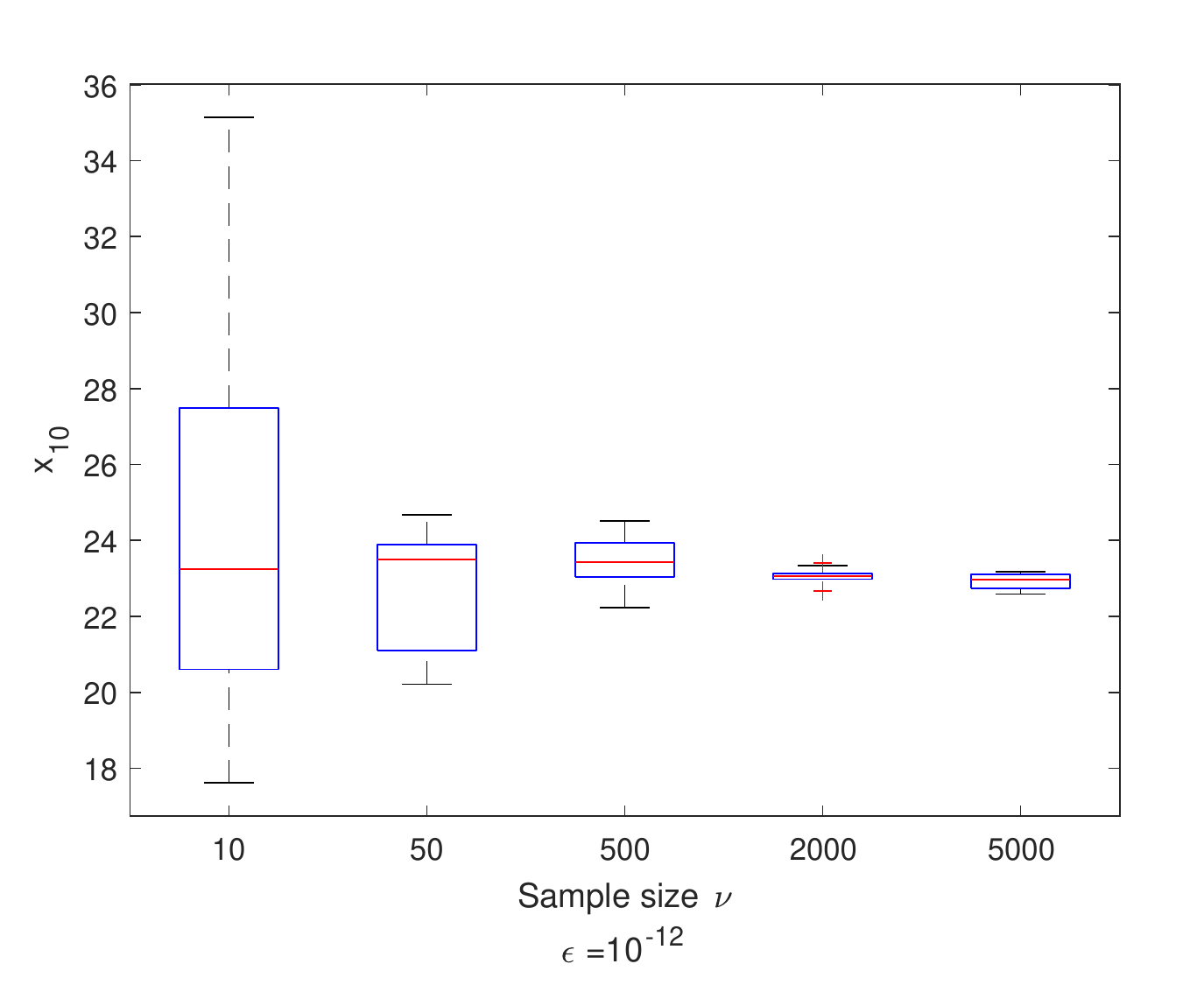}}
   \caption{Convergence property of $x$ with increasing $\nu$,  $J=10$.}
   \label{figure:player-10}
   \end{figure}

\subsection{Crude oil market}
In this subsection, we demonstrate one application of our two-stage stochastic Nash equilibrium model in describing the crude oil market share.
Namely, we investigate the strategies of crude oil exporting agents via solutions of our reformulated SVI problem so as to recreate actual market shares.
Based on historical data on crude oil market, we make in-sample back-tracking test to establish the effectiveness and validity of our model while explaining the market behaviour.
Furthermore, the out-of-sample prediction capability of our model is demonstrated when training data is used to specify model parameters.
From the results of our numerical tests, we conclude that our model is suitable to reproduce, predict and potentially capable to explain stable market shares of crude oil.

\begin{figure}[!htbp]
  \centering
    {\includegraphics[height=6.8cm]{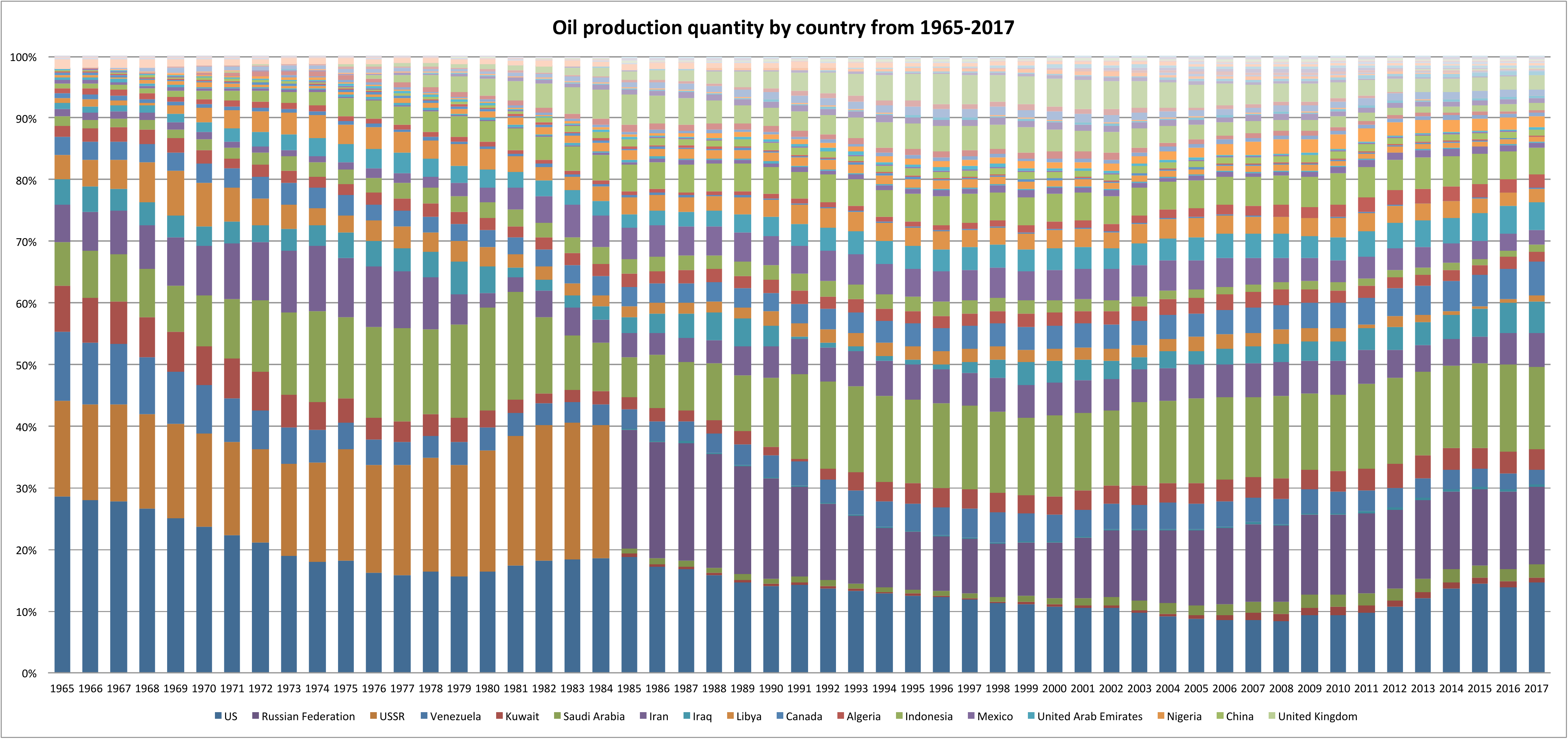}}
  \caption{Market shares of different oil-producing countries, 1965-2017.}
\label{market-share-2008-2017}
\end{figure}

Crude oil market is one of the most widely studied commodity market in the world.
The market has long been described as of an `intermediate between monopoly and perfect competition \cite{H1931economics}'', supported by extensive historical data and market observation.
When the market encounters large sudden events, it took ``oil shocks'' \cite{H1996happened,H2003oil,K2009not}.
Conceptually, consequences of supply and demand fluctuation, along with many other factors, are reflected most directly in dramatic changes in oil price.
One interesting observation is that the market share behaves rather smoothly even during periods of oil shocks \cite{K2009not}.
Majority of the world's crude oil is supplied by a few large oil exporting countries and they are viewed collectively as a finite number of large agents from which price-taking consumers purchase product at the same price \cite{S2009world}.
Hence, it is no surprise that CN approaches have been adopted from earlier days of study in this field, see \cite{S1976exhaustible}.
The study of oil shock gives rise to rich field of economic researches in oil market structure and as well as factors affecting oil price, see \cite{KM2014role,JP2015speculation,BLR2017lags}, etc.
Our model does not attempt to make explanation or future predictions on oil shocks but rather aims to make sense of the stable characteristics of supply-side of oil market share: major agents acting non-cooperatively to achieve market equilibrium.
We focus on the supply factor of oil, while treat other factors, e.g., demand, world economic situation, population, etc., as known information with uncertainty.
Our model is found to be able to reproduce the market shares as well as to forecast future production plans of different oil-producing countries based on historical data.
In short, we firstly use historical data to determine model parameters as well as to approximate distribution for uncertainty in observations of the market.
Model parameters are tested within a fixed sampling window so that the in-sample results obtained by solving our model matches that of training data set.
It then follows that over a short decision horizon in the future, we can adopt the trained model to prediction future production plans.
More specifically, we took i.i.d. samples using a fixed-size rolling window sampling method while the prediction is obtained by solving our model with the adjusted set of parameters obtained by in-sample tests within the training window.

\subsubsection{Data set, parameter selection and uncertainty description}
In order to implement real market data into our model (\ref{M1})-(\ref{M2}), we need to determine parameters $c_j, a_j$ correspond to production costs of $j$-th agent by analyzing training set.
In practice, model parameters are adjusted via a brute force learning process.
More specifically, if the training set consists only one trading day of information, we adjust the parameters so that the solution, i.e., production quantities of $J$ agents, matches the estimated historical observation.
When more data become available, the parameters are adjusted so that the average of solutions is close to the real estimate.
It is found from our numerical implementation that the best fit solutions are obtained when values of $c_j$ and $a_j$ are taken to be inversely proportional to known market share within the sampling period.
It is firstly learned that the common belief of ``lower unit cost would lead to large sale'' does not hold in our model, which also motivates the choice of quadratic production cost in our model.
Moreover, it is found in our numerical experiment that the smaller the quadratic cost parameters $c_j$ the greater market share the agent would have.
This can be understood from the fact that the large oil producing agents are more flexible in adjusting their production quantities as compare to smaller agents.
A more mathematical interpretation is to think of $c_j$ as related to the penalty parameters of the augmented Lagrangian formulation corresponding to production constraints of each agent: production cost is upper bounded with some unit production cost $\tilde{a}_j$.
In this way, one can think of this observation been that the countries has greater reserve of oil suffers less when the production upper bound gets violated.
Nevertheless, we found that the quadratic is more than adequate to reproduce historical observations.

The scenarios of stochastic parameters are taken from empirical distribution of the sampling data set, guided by the analysis of \cite{H2011historical}.
In detail, we draw samples to generate scenarios of market's supply discounting factor $\gamma(\xi)$ and $j$-th agent's adjusted price $p_j(\xi)=p_0(\xi)-h_j(\xi)$.
The data set used in this study are from (i)\emph{Statistical Review of World Energy}\footnote{https://www.bp.com/en/global/corporate/energy-economics/statistical-review-of-world-energy.html}, yearly data from 1984, published every June by bp. Inc; (ii) \emph{Oil Price Dynamics Report}\footnote{https://www.newyorkfed.org/research/policy/oil\_price\_dynamics\_report}, weekly by Federal Reserve Bank of New York since 1986; (iii)\emph{U.S. Energy Information Administration}\footnote{https://www.eia.gov}, weekly and daily spot price of Brent.
The empirical distribution are generated from the contents of oil dynamic report, consists of time series of weekly percentage change of Brent spot price and its corresponding components' contributions, i.e., those from demand, residual, and supply.
More specifically, the percentage change in price in $(k+1)$-th\footnote{$k$ is used to denote the order in time, distinguished from that of i.i.d. samples} week compared to that of the $k$-th week is denoted by $\triangle p^k$, while the contribution of demand and supply changes are denoted as $\triangle D^k$ and $\triangle S^k$ respectively.
The contribution to the price change that does not correspond to demand and supply is denoted as residual $\triangle R^k$.
Therefore, for any week within the report, it holds that the Brent price change is determinately represented by
\begin{equation}\label{pricebreakdown}
\triangle p^k = \triangle D^k + \triangle T^k + \triangle R^k.
\end{equation}
Recall that our expression for (scenario-based) inverse demand function of quantity offered to the market for a given scenario $\xi_\ell$ is in the following form
\[
p^k(\xi_\ell)-\gamma^k(\xi_\ell)T^k(y_{\xi_\ell}).
\]
In accordance with data structure, $p^k(\xi_\ell)$ corresponds to the adjusted predicted price from contributions other than supply, i.e.,
\[
p^k(\xi_\ell) = p_0^k\big(1 + \triangle d^k(\xi_\ell)+\triangle r^k(\xi_\ell)\big),
\]
where $p_0^k$ denotes the known Brent spot price prior to that of the concerned time.
Both $\triangle d^k(\xi_\ell)$ and $\triangle r^k(\xi_\ell)$ are random scenarios taken from empirical distributions of historical demand and residual distributions within the sampling set respectively.
To estimate $\gamma^k(\xi)$, corresponding to $k$-week data, we need the data of spot quantity supplied to the market, $T^k$.
In practice, it is very difficult to obtain reliable data on total supply to the market over short observation window.
Rather, we observe that trust-worthy estimate on daily oil supply based on annual data is available along with the observation that a steady growth of about 1\% per year over the last four decades, according to \emph{Statistical Review of World Energy}.
Therefore, we take random instances of daily market supply $T^k(y_\xi)$ which are taken from a uniformly distributed interval between 99\% and 101\% of yearly based daily estimate.
Then, we can generate a set of data of stochastic excessive supply discount factor $\{\gamma^k(\xi_\ell)\}$,
\[\gamma^k(\xi_\ell)=\frac{|p^k(\xi_\ell)-p_0^k|}{T^k(y_{\xi})},\]
where $p_0^k$ is known with certainty within the testing data set, and the $|\cdot|$ ensures that increase in quantity has a negative influence on price.

We present the results for our numerical experiments on reproduced oil market share.
As shown in Fig \ref{market-share-2008-2017}, the in-sample results over the periods of oil shocks 2007, 2009 and 2014 are reproduced using solely from the available information within that year.
For example, the in-sample experiment of 2009 uses all the available data obtained within 2009, e.g., daily price from 01/01/2009 to 31/12/2009.
For the generation of random instances in-sample, we use weekly price contribution covering the year 2009 to form empirical distributions for supply excluded spot prices $p(\xi_k)$, and discounting factor $\gamma(\xi_k)$, from which we took i.i.d. samples to represent random scenarios.
Hence, for in-sample experiments, our goal is to reproduce known market share with the adjustments of cost parameters $a_j,c_j$ corresponding to $j$-th oil producing agent.

\begin{figure}[!htbp]
  \centering
    {\includegraphics[height=6cm]{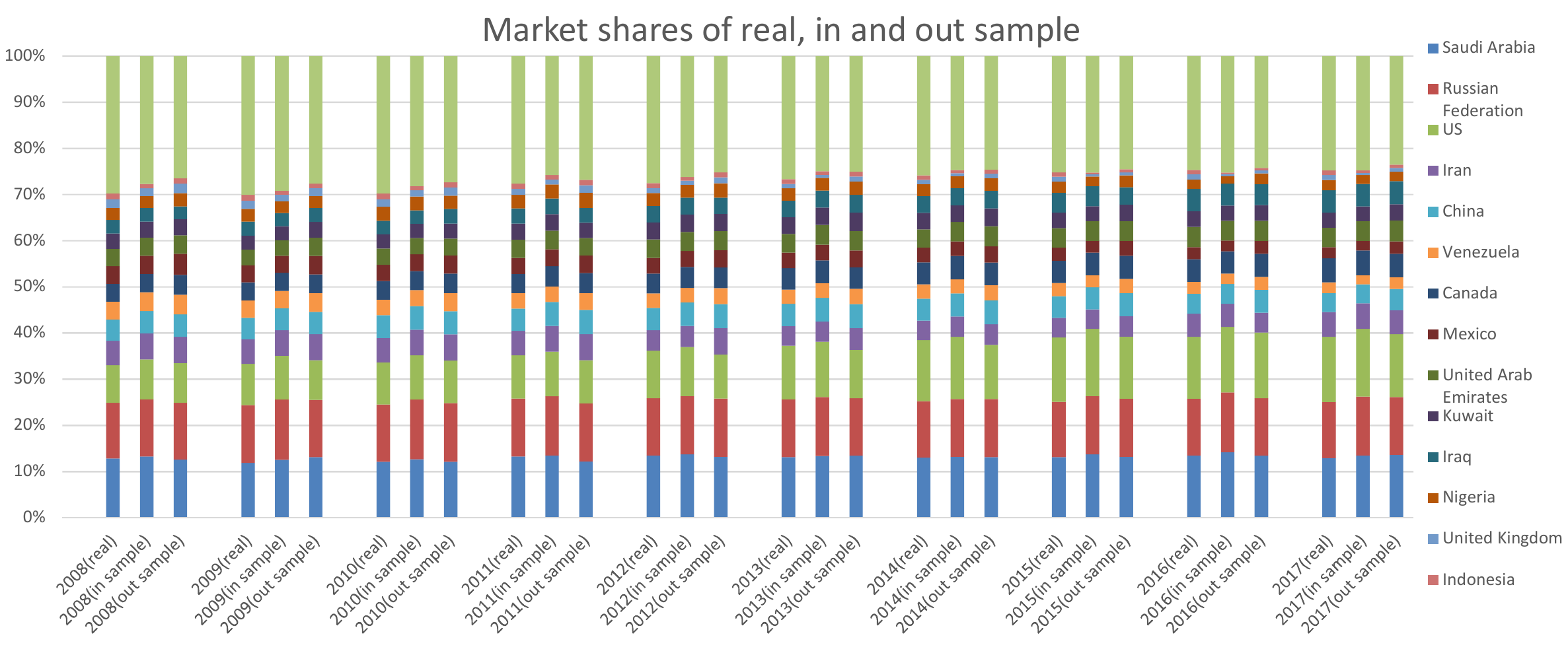}}
  \caption{Real, in-sample and out-of-sample market shares results, 2008-2017.}
\label{market-share-2008-2017}
\end{figure}

The out-of-sample tests cover the period from 2007 to 2017 with the \emph{one year length equivalent of sampling window} in Figure \ref{market-share-2008-2017}.
Note that one year sampling window is adopted since independent numerical results of selected years demonstrated that the market has a short memory in the sense that longer historical data would not provide extra information, at least in the setup of our model and approaches.
Prior to the date of interests, all the historical data are assumed to be available while the rolling window sampling permits the usage of newly acquired data as soon as it becomes available.
From the construction and interpretation of our two-stage model, it is reasonable to assume that cost of oil production remain unchanged over a short decision horizon.
Furthermore, we tested different lengths of the sampling window, especially during times of oil shocks for better in-sample fitting, and the best fit sampling window length is used in out-of-sample tests.
For example, we are interested in obtaining the market share of 2009, during times of recent financial crisis, using the fixed length sampling windows whose length varies from one year to five years.
If we want to predict the production on 01/01/2009 using a sampling window of one year, we took one year sampling window from 01/01/2008 to 31/01/2008 to obtain empirical distributions from which scenarios of $\gamma(\xi_k)$, $p(\xi_k)$ are sampled.
Taking $\nu$ samples, the solutions method follows that described in previous subsection.
For the result on 02/01/2019, the sampling window rolls forward using the data form 02/01/2008 to 01/01/2009 since the last day information becomes known, while the model remains unchanged.
Therefore, for the market share prediction of any given year, we obtained the daily results by taking average of 250 samples, while the yearly estimate is formed by taking the average on daily results.
From Figure \ref{market-share-2008-2017}, it is observed that we can obtain fair good prediction quantities through out the testing period.
However, it is worth mentioning that during the numerical experiments across that decision horizon the computational cost is much higher during periods over oil shocks compare to those of relative stable oil prices.

\section{Conclusion}
\label{sec:6}
A two-stage SVI is used to describe the equilibrium of a convex two-stage non-cooperative multi-agent game under uncertainty.
Sufficient conditions for the existence of solutions of the two-stage SVI are provided.
The numerical implementation is constructed by proposing a regularized SAA method.
Furthermore, we prove the convergence of the regularized SAA method as the regularization parameter tends to zero and the sample size tends to infinity.
A real application of crude oil market equilibrium is studied using our approach.
More specifically, the oligopolistic market is characterized by a two-stage stochastic production and supply planning game and the market share is achieved at the equilibrium.
Numerical results are presented based on historical data over the last fifty years and the effectiveness of our two-stage SVI approach of market share is demonstrated.

\begin{appendices}
\label{Appendix}
\section{Appendix}
Proof of Theorem \ref{th:cf}.
\begin{proof}
By direct computation, we have that
\begin{align*}
&\left(
M^\epsilon(\xi)
\begin{pmatrix}
y^\epsilon\\
\lambda^\epsilon
\end{pmatrix}
+
q(x,\xi)\right)_j\\
&~~~~~~=
\begin{cases}
\gamma(\xi) (y^\epsilon)_j + \gamma(\xi)T^\epsilon  +(\lambda^\epsilon)_j -p_j(\xi), ~ j=1,\ldots,J;\\
x_{j-J} - (y^\epsilon)_{j-J} + \epsilon (\lambda^\epsilon)_{j-J}, ~j=J+1,\ldots,2J.
\end{cases}
\end{align*}
Then, we can rewrite problem \eqref{RM} as below:
\begin{align}
\label{gs3}
\begin{cases}
0\leq (y^\epsilon)_j            \bot \gamma(\xi) (y^\epsilon)_j + \gamma(\xi)T^\epsilon +(\lambda^\epsilon)_j -p_j(\xi) \geq 0,\\
0\leq (\lambda^\epsilon)_j \bot     x_j - (y^\epsilon)_j         +                \epsilon (\lambda^\epsilon)_j                         \geq 0,
\end{cases}
\end{align}
for $j\in\mathcal{J}$.
From the first complementarity condition in \eqref{gs3}, we have $(y^\epsilon)_j$ as follows:
\begin{align}
\label{gs4}
(y^\epsilon)_j=
\begin{cases}
\displaystyle-\frac{\gamma(\xi)T^\epsilon + (\lambda^\epsilon)_j - p_j(\xi)}{\gamma(\xi)}, & \gamma(\xi) T^\epsilon + (\lambda^\epsilon)_j - p_j(\xi) < 0;\\
0, & \gamma(\xi)T^\epsilon + (\lambda^\epsilon)_j - p_j(\xi) \geq 0
\end{cases}
\end{align}
for $j\in\mathcal{J}$.
Similarly, we can derive that
\begin{align}
\label{gs5}
(\lambda^\epsilon)_j=
\begin{cases}
\displaystyle\frac{(y^\epsilon)_j - x_j}{\epsilon}, & (y^\epsilon)_j - x_j > 0;\\
0, &  (y^\epsilon)_j - x_j \leq 0
\end{cases}
\end{align}
for $j\in\mathcal{J}$. Note that $(y^\epsilon)_j=0$ implies $(y^\epsilon)_j=0\leq x_j$, and we have $(\lambda^\epsilon)_j=0$.
Then, based on \eqref{gs4} and \eqref{gs5}, we have for all three cases:
\begin{equation*}
\begin{cases}
(y^\epsilon)_j=0,~(\lambda^\epsilon)_j=0&~\mathrm{for}~j\in\mathcal{I}_1;\\
\displaystyle(y^\epsilon)_j=-\frac{\gamma(\xi) T^\epsilon + (\lambda^\epsilon)_j - p_j(\xi)}{\gamma(\xi)},~(\lambda^\epsilon)_j=                         0                              &~\mathrm{for}~j\in\mathcal{I}_2;\\
\displaystyle(y^\epsilon)_j=-\frac{\gamma(\xi) T^\epsilon + (\lambda^\epsilon)_j - p_j(\xi)}{\gamma(\xi)},~(\lambda^\epsilon)_j= \frac{(y^\epsilon)_j-x_j}{\epsilon} &~\mathrm{for}~j\in\mathcal{I}_3,
\end{cases}
\end{equation*}
where
\begin{align*}
\mathcal{I}_1:=\left\{j\in\mathcal{J}: \gamma(\xi)T^\epsilon + (\lambda^\epsilon)_j - p_j(\xi) \geq 0, (y^\epsilon)_j - x_j \leq 0 \right\},\\
\mathcal{I}_2:=\left\{j\in\mathcal{J}: \gamma(\xi)T^\epsilon + (\lambda^\epsilon)_j - p_j(\xi)   <    0, (y^\epsilon)_j - x_j \leq 0 \right\},\\
\mathcal{I}_3:=\left\{j\in\mathcal{J}: \gamma(\xi)T^\epsilon + (\lambda^\epsilon)_j - p_j(\xi)   <    0, (y^\epsilon)_j - x_j   >   0 \right\}.
\end{align*}
It follows that,
\begin{equation*}
\begin{aligned}
\left((y^\epsilon)_j,(\lambda^\epsilon)_j\right)
=
\begin{cases}
(0,0),                                                                                                                                                                                                                                                            & j\in\mathcal{I}_1;\\
\displaystyle\left(-\frac{\gamma(\xi) T^\epsilon  - p_j(\xi)}{\gamma(\xi)},0\right),                                                                                                                                                             & j\in\mathcal{I}_2;\\
\displaystyle\left(-\frac{\epsilon \gamma(\xi)T^\epsilon - x_j - \epsilon p_j(\xi)}{\epsilon\gamma(\xi) +1}, -\frac{\gamma(\xi)(T^\epsilon+x_j)- p_j(\xi)}{\epsilon\gamma(\xi)+1} \right), & j\in\mathcal{I}_3,
\end{cases}
\end{aligned}
\end{equation*}
which verifies \eqref{CFS}.
For the remaining of the proof, let $j\in\mathcal{I}_2$, we have
\begin{align*}
-\gamma(\xi)(y^\epsilon)_j=\gamma(\xi) T^\epsilon -p_j(\xi)
\end{align*}
and thus
\begin{align}
\label{gs11}
-\gamma(\xi)\sum_{i\in\mathcal{I}_2}(y^\epsilon)_i=|\mathcal{I}_2|\gamma(\xi) T^\epsilon - \sum_{i\in\mathcal{I}_2}p_i(\xi).
\end{align}
Analogously, we can derive from
$$(y^\epsilon)_j=-\frac{\gamma(\xi) T^\epsilon + (\lambda^\epsilon)_j -p_j(\xi)}{\gamma(\xi)}~\mathrm{and}~(\lambda^\epsilon)_j=\frac{(y^\epsilon)_j-x_j}{\epsilon}~\mathrm{for}~j\in\mathcal{I}_3$$
that
\begin{align}
\label{gs12}
-\gamma(\xi) \sum_{i\in\mathcal{I}_3}(y^\epsilon)_i=|\mathcal{I}_3|\gamma(\xi) T^\epsilon + \frac{1}{\epsilon}\sum_{i\in\mathcal{I}_3}(y^\epsilon)_i - \frac{1}{\epsilon}\sum_{i\in\mathcal{I}_3}x_i - \sum_{i\in\mathcal{I}_3}p_i(\xi).
\end{align}
Combining that of \eqref{gs11} and \eqref{gs12}, we obtain
\begin{align*}
-\gamma(\xi)T^\epsilon=(|\mathcal{I}_2|+|\mathcal{I}_3|)\gamma(\xi) T^\epsilon + \frac{1}{\epsilon}\sum_{i\in\mathcal{I}_3}(y^\epsilon)_i - \frac{1}{\epsilon}\sum_{i\in\mathcal{I}_3}x_i- \sum_{i\in\mathcal{I}_2\cup\mathcal{I}_3}p_i(\xi).
\end{align*}
Therefore, we have
\begin{align}
\label{gs13}
\frac{1}{\epsilon}\sum_{i\in\mathcal{I}_3}(y^\epsilon)_i=-(\abs{\mathcal{I}_2}+|\mathcal{I}_3|+1)\gamma(\xi) T^\epsilon +  \frac{1}{\epsilon}\sum_{i\in\mathcal{I}_3}x_i + \sum_{i\in\mathcal{I}_2\cup\mathcal{I}_3}p_i(\xi).
\end{align}
Substituting \eqref{gs13} into \eqref{gs12}, we have derivation
\begin{align*}
 \abs{\mathcal{I}_3}\gamma(\xi) T^\epsilon =& -\left(\gamma(\xi)+\frac{1}{\epsilon}\right) \sum_{i\in\mathcal{I}_3}(y^\epsilon)_i + \frac{1}{\epsilon}\sum_{i\in\mathcal{I}_3}x_i +  \sum_{i\in\mathcal{I}_3} p_i(\xi)\\
=& -\left(\epsilon\gamma(\xi) +1\right) \left( -(\abs{\mathcal{I}_2}+\abs{\mathcal{I}_3}+1)\gamma(\xi) T^\epsilon +  \frac{1}{\epsilon}\sum_{i\in\mathcal{I}_3}x_i + \sum_{i\in\mathcal{I}_2\cup\mathcal{I}_3} {p(\xi)}_i \right) \\
&+ \frac{1}{\epsilon}\sum_{i\in\mathcal{I}_3}x_i + \sum_{i\in\mathcal{I}_3} p_i(\xi) \\
=& \left(\epsilon\gamma(\xi)+1\right)(\abs{\mathcal{I}_2}+\abs{\mathcal{I}_3}+1)\gamma(\xi) T^\epsilon - \gamma(\xi) \sum_{i\in\mathcal{I}_3}x_i - \epsilon\gamma(\xi) \sum_{i\in\mathcal{I}_2\cup\mathcal{I}_3}p_i(\xi) - \sum_{i\in\mathcal{I}_2} p_i(\xi).
\end{align*}
Then, we get
\begin{align*}
\left(\epsilon\gamma(\xi)(\abs{\mathcal{I}_2}+\abs{\mathcal{I}_3}+1) + \abs{\mathcal{I}_2} + 1 \right) \gamma(\xi) T^\epsilon = & \gamma(\xi)\sum_{i\in\mathcal{I}_3}x_i + \epsilon\gamma(\xi)\sum_{i\in\mathcal{I}_2\cup\mathcal{I}_3}p_i(\xi)+ \sum_{i\in\mathcal{I}_2}p_i(\xi),
\end{align*}
that is,
\begin{align*}
T^\epsilon =  \frac{\gamma(\xi) \sum_{i\in\mathcal{I}_3}x_i + \epsilon\gamma(\xi) \sum_{i\in\mathcal{I}_2\cup\mathcal{I}_3}p_i(\xi) + \sum_{i\in\mathcal{I}_2} p_i(\xi)}{\left(\epsilon\gamma(\xi)(\abs{\mathcal{I}_2}+\abs{\mathcal{I}_3}+1) + \abs{\mathcal{I}_2} + 1 \right) \gamma(\xi)}.
\end{align*}
This completes the proof.
\qed\end{proof}

Proof of Proposition \ref{p2}.
\begin{proof}
Let $\hat{z}=(\hat{y},\hat{\lambda})$ be any solution of LCP$(q,M)$ and we have the derivation:
\begin{align*}
0 &\geq (z^\epsilon - \hat{z})^T (M^\epsilon z^\epsilon +q -  (M\hat{z} +q) )\\
&=(z^\epsilon - \hat{z})^T (M^\epsilon z^\epsilon - M \hat{z} )\\
&=(z^\epsilon - \hat{z})^T M (z^\epsilon - \hat{z}) + (z^\epsilon - \hat{z})^T \begin{pmatrix}0 \\ \epsilon \lambda^\epsilon \end{pmatrix}\\
&\geq (z^\epsilon - \hat{z})^T \begin{pmatrix}0 \\ \epsilon \lambda^\epsilon \end{pmatrix}\\
&= \epsilon (\lambda^\epsilon - \hat{\lambda})^T \lambda^\epsilon,
\end{align*}
where the second inequality follows from the positive semidefiniteness of $M$.
Then, we have
\begin{align*}
\Vert \lambda^\epsilon\Vert^2 \leq  \hat{\lambda}^T \lambda^\epsilon \leq \Vert \hat{\lambda}\Vert \Vert\lambda^\epsilon\Vert,
\end{align*}
which implies the boundedness of $\lambda^\epsilon$,
\begin{align}
\label{gs9}
\Vert \lambda^\epsilon\Vert \leq  \Vert \hat{\lambda}\Vert.
\end{align}
It follows from \eqref{gs9} that any accumulation point of $\{\lambda^\epsilon\}$ as $\epsilon \downarrow 0$ is the least $l_2$-norm solution. Since $M$ is positive semidefinite, we know from \cite[Theorem 5.6.2]{RPS1992the} that there is a unique least $l_2$-norm solution. On the other hand, we know from Proposition \ref{p1}, for any fixed $(x,\xi)$, $\hat{y}$ is unique. Therefore, the limit of $z^\epsilon$ exists as $\epsilon\downarrow0$ and converges to the least $l_2$-norm solution of LCP$(q,M)$.

Due to the existence of limit for $z^\epsilon$ as $\epsilon\downarrow0$, \eqref{LNS} can be derived directly from \eqref{CFS}. In what follows, we focus on deriving the expression \eqref{gs7}. To this end, for each $j\in\mathcal{J}$, three cases are discussed:
\begin{align}
\gamma(\xi) T^\epsilon + (\lambda^\epsilon)_j - p_j(\xi) \geq 0,~ (y^\epsilon)_j - x_j \leq 0,\label{gs16}\\
\gamma(\xi) T^\epsilon + (\lambda^\epsilon)_j - p_j(\xi) < 0,~ (y^\epsilon)_j - x_j \leq 0,\label{gs17}\\
\gamma(\xi) T^\epsilon + (\lambda^\epsilon)_i - p_j(\xi) < 0,~ (y^\epsilon)_j - x_j > 0.\label{gs18}
\end{align}

\textbf{Case 1:} If there exists a sequence $\{\epsilon_k\}_{k=1}^\infty$ converging to $0$ such that \eqref{gs16} holds, we have
$$\lim_{k\rightarrow\infty}\big((y^{\epsilon_k})_j, (\lambda^{\epsilon_k})_j\big)=(0,0).$$
Thus, $\abs{(\lambda^{\epsilon_k})_j - \bar{\lambda}_j} = 0$.

\textbf{Case 2:} If there exists a sequence $\{\epsilon_k\}_{k=1}^\infty$ converging to $0$ such that \eqref{gs17} holds, we have an estimation
\begin{align*}
\lim_{k\rightarrow\infty}\big((y^{\epsilon_k})_j, (\lambda^{\epsilon_k})_j\big)&= \lim_{k\rightarrow\infty}\left(-\frac{\gamma(\xi) T^{\epsilon_k} - p_j(\xi)}{\gamma(\xi)},0\right)\\
&=\left(-\frac{\gamma(\xi) \lim_{k\rightarrow\infty}T^{\epsilon_k} - p_j(\xi)}{\gamma(\xi)},0\right)\\
&=\left(-\frac{\gamma(\xi) \bar{T} - p_j(\xi)}{\gamma(\xi)},0\right).
\end{align*}
Thus, $\abs{(\lambda^{\epsilon_k})_j - \bar{\lambda}_j}= 0$.

\textbf{Case 3:} If there exists a sequence $\{\epsilon_k\}_{k=1}^\infty$ converging to $0$ such that \eqref{gs18} holds, we have
\begin{align*}
\lim_{k\rightarrow\infty}\big((y^{\epsilon_k})_j, (\lambda^{\epsilon_k})_j\big)&= \lim_{k\rightarrow\infty}\left(-\frac{ {\epsilon_k}\gamma(\xi) T^{\epsilon_k} - x_j -{\epsilon_k} p_j(\xi)}{{\epsilon_k}\gamma(\xi) +1}, -\frac{\gamma(\xi) (T^{\epsilon_k}+x_j) - p_j(\xi)}{{\epsilon_k}\gamma(\xi)+1} \right)\\
&=\left( x_j, -\gamma(\xi) (\bar{T}+x_j) + p_j(\xi) \right).
\end{align*}
Thus, we have
\begin{align*}
&~~~~\abs{(\lambda^{\epsilon_k})_j-\bar{\lambda}_j}\\
&=\abs{ (\lambda^{\epsilon_k})_j + \gamma(\xi)(\bar{T}+x_j) - p_j(\xi)} \\
&= \abs{ -\frac{\gamma(\xi)(T^{\epsilon_k}+x_j)  - p_j(\xi)}{{\epsilon_k}\gamma(\xi)+1}  +  \gamma(\xi)(\bar{T} + x_j) - p_j(\xi) }\\
&= \frac{ \abs{-\gamma(\xi)(T^{\epsilon_k}+x_j) + p_j(\xi) +\gamma(\xi)(\bar{T} + x_j) - p_j(\xi) + {\epsilon_k}\gamma(\xi) (\gamma(\xi)(\bar{T}+x_j) - p_j(\xi))} }{{\epsilon_k}\gamma(\xi)+1}\\
&\leq\frac{\gamma(\xi)\abs{T^{\epsilon_k}-\bar{T}} + \abs{ \gamma(\xi) (\gamma(\xi)(\bar{T} +x_j) - p_j(\xi))}\epsilon_k  }{{\epsilon_k}\gamma(\xi)+1}.
\end{align*}
Collectively, we know from \textbf{Case 1}, \textbf{Case 2} and \textbf{Case 3} that
\begin{align}
  (y^{\epsilon_k})_j-\bar{y}_j&=0,\label{gs16-1}\\
  (y^{\epsilon_k})_j-\bar{y}_j&= -(T^{\epsilon_k}-\bar{T}),\label{gs17-1}\\
  (y^{\epsilon_k})_j-\bar{y}_j&= \frac{-\gamma(\xi) T^{\epsilon_k} + p_j(\xi) -\gamma(\xi) x_j}{{\epsilon_k}\gamma(\xi) +1}\cdot \epsilon_k.\label{gs18-1}
\end{align}
Furthermore, we have that $T^{\epsilon_k}-\bar{T}\geq0$ always holds.
For the purpose of arriving at a contradiction, we assume $T^{\epsilon_k}-\bar{T} < 0$. Then \eqref{gs17-1} implies that
$$y^{\epsilon_k})_j-\bar{y}_j> 0.$$
Moreover, \eqref{gs16-1} and \eqref{gs18-1} induce
\begin{align*}
(y^{\epsilon_k})_j-\bar{y}_j &= 0,\\
(y^{\epsilon_k})_j-\bar{y}_j &\geq x_j - \bar{y}_j \geq  0,
\end{align*}
respectively. Clearly, we have $T^{\epsilon_k}-\bar{T}\geq 0$, which contradicts our assumption.
In addition, we have
\begin{align*}
T^{\epsilon_k}-\bar{T} &\leq  \frac{-\gamma(\xi) T^{\epsilon_k} + \norm{p(\xi)}_1 + \gamma(\xi) \norm{x}_1 }{{\epsilon_k} \gamma(\xi)+1}\cdot \epsilon_k
\leq  \left( \norm{p(\xi)}_1 + \gamma(\xi) \norm{x}_1 \right) \epsilon_k.
\end{align*}
Then, it follows that
\begin{align*}
&~~~~\abs{(\lambda^{\epsilon_k})_j-\bar{\lambda}_j}\\
&\leq\frac{\gamma(\xi) \abs{T^{\epsilon_k}-\bar{T}} + \abs{ \gamma(\xi) (\gamma(\xi)(\bar{T}+x_j) - p_j(\xi))}\epsilon_k  }{{\epsilon_k}\gamma(\xi)+1} \\
&\leq
\frac{\gamma(\xi) \left( \norm{p(\xi)}_1 + \gamma(\xi) \norm{x}_1 \right) + \abs{ \gamma(\xi) (\gamma(\xi)(\bar{T}+x_j) - p_j(\xi))} }{{\epsilon_k}\gamma(\xi)+1} \cdot \epsilon_k \\
&\leq \left(\gamma(\xi) \left( \norm{p(\xi)}_1 + \gamma(\xi) \norm{x}_1 \right)  +  \gamma(\xi)^2\left(\norm{x}_1 + \frac{\norm{p(\xi)}_1}{\gamma(\xi)}+ \norm{x}_1\right) + \gamma(\xi)\norm{p(\xi)}_1\right) \epsilon_k\\
&\leq 3\left( \gamma(\xi)^2 \norm{x}_1 + \gamma(\xi)\norm{p(\xi)}_1 \right) \epsilon_k,
\end{align*}
where the third inequality follows Lemma \ref{lem2} and the continuity of $T^\epsilon$ that
\begin{align*}
\bar{T} \leq \norm{x}_1 + \frac{\norm{p(\xi)}_1}{\gamma(\xi)}.
\end{align*}
To summarize, for each $j\in\mathcal{J}$, we always have
\begin{equation*}
\abs{(\lambda^\epsilon)_j-\bar{\lambda}_j}\leq 3\left( \gamma(\xi)^2 \norm{x}_1 + \gamma(\xi)\norm{p(\xi)}_1 \right) \epsilon.
\end{equation*}
Then, according to the definition of $l_2$-norm, for any given $x\in\mathbb{R}_+^J$ one can compute
$$\bar{\kappa}(\xi):=3\sqrt{J}\left( \gamma(\xi)^2 \norm{x}_1 + \gamma(\xi) \norm{p(\xi)}_1 \right).$$
\qed\end{proof}
  \end{appendices}


\begin{thebibliography}{}
\bibitem{BLR2017lags} G.~Bornstein, P.~Krusell and S.~Rebelo, Lags, costs, and shocks: An equilibrium model of the oil industry, {\it National Bureau of Economic Research}, 2017.
\bibitem{BCS2017subdifferentiation} J.V.~Burke, X.~Chen and H.~Sun, Subdifferentiation and smoothing of nonsmooth integral functionals, Department of applied mathematics, The Hong Kong Polytechnic University, 2017.
\bibitem{C2012smoothing} X.~Chen, Smoothing methods for nonsmooth, nonconvex minimization, {\it Math. Program.,} {\bf 134} (2012), 71-99.
\bibitem{CF2004smoothing} X.~Chen and M.~Fukushima, A smoothing method for a mathematical program with P-matrix linear complementarity constraints,  {\it Comput. Optim. Appl.,} {\bf 27} (2004), 223-246.
\bibitem{CF2005expected} X.~Chen and M.~Fukushima, Expected residual minimization method for stochastic linear complementarity problems,  {\it Math. Oper. Res.,} {\bf 30} (2005), 1022-1038.
\bibitem{CPW2017two} X.~Chen, T.K.~Pong and R.J.B.~Wets, Two-stage stochastic variational inequalities: an ERM-solution procedure,  {\it Math. Program.,} {\bf 165} (2017), 71-111.
\bibitem{CSW2015regularized} X.~Chen, H.~Sun and R.J.B.~Wets, Regularized mathematical programs with stochastic equilibrium constraints: estimating structural demand models,  {\it SIAM J. Optim.,} {\bf 25} (2015), 53-75.
\bibitem{CSS2018convergence} X.~Chen, A.~Shapiro and H.~Sun, Convergence analysis of sample average approximation of two-state stochastic generalized equations,  {\it SIAM J. Optim.,} {\bf 29} (2019), 135-161.
\bibitem{CSX2017discrete} X.~Chen, H.~Sun and H.~Xu, Discrete approximation of two-stage stochastic and distributionally robust linear complementarity problems,  {\it Math. Program.}, (2018), to appear.
\bibitem{CWZ2012stochastic} X.~Chen, R.J.B.~Wets and Y.~Zhang, Stochastic variational inequalities: residual minimization smoothing sample average approximations,  {\it SIAM J. Optim.,} {\bf 22} (2012), 649-673.
\bibitem{CX2013newton} X.~Chen and S.~Xiang, Newton iterations in implicit time-stepping scheme for differential linear complementarity systems,  {\it Math. Program.,} {\bf 138} (2013), 579-606.
\bibitem{CY1999homotopy} X.~Chen and Y.~Ye, On homotopy-smoothing methods for box-constrained variational inequalities,  {\it SIAM J. Control. Optim.,} {\bf 37} (1999), 589-616.
\bibitem{CY2000smoothing} X.~Chen and Y.~Ye, On smoothing methods for the $P_0$ matrix linear complementarity problem,  {\it SIAM J. Optim.,} {\bf 11} (2000), 341-363.
\bibitem{FP2007finite} F.~Facchinei and J.S.~Pang,  {\it Finite-dimensional variational inequalities and complementarity problems}, Springer Science \& Business Media, 2007.
\bibitem{GM1997smoothing} S.A.~Gabriel and J.J.~More, Smoothing of mixed complementarity problems, in  {\it Complementarity and Variational Problems: State of the Art, M.C. Ferris and J.S. Pang, eds.,} SIAM, (1997), 105-116.
\bibitem{GRS2007dynamic} T.S.~Genc, S.S.~Reynolds and S.~Sen, Dynamic oligopolistic games under uncertainty: A stochastic programming approach,  {\it J. Econ. Dyn. Control.,} {\bf 31} (2007), 55-80.
\bibitem{H1996happened} J.D.~Hamilton, This is what happened to the oil price-macroeconomy relationship,  {\it J. Monet. Econ.,} {\bf 38} (1996), 215-220.
\bibitem{H2003oil} J.D.~Hamilton, What is an oil shock?,  {\it J. Econom.,} {\bf 113}:2 (2003), 363-398.
\bibitem{H2011historical} J.D.~Hamilton, Historical oil shocks, \emph{National Bureau of Economic Research}, 2011.
\bibitem{HP1990finite} P.T.~Harker and J.S.~Pang, Finite-dimensional variational inequality and nonlinear complementarity problems: a survey of theory, algorithms and applications,  {\it Math. Program.,} {\bf 48} (1990), 161-200.
\bibitem{HP2007nash} B.F.~Hobbs and J.S.~Pang, Nash-Cournot equilibria in electric power markets with piecewise linear demand functions and joint constraints, \emph{Oper. Res.,} \textbf{55} (2007), 113-127.
\bibitem{H1931economics}  H.~Hotelling, The economics of exhaustible resources,  {\it J. Polit. Econ.,} {\bf 39} (1931), 137-175.
\bibitem{HLP2013demand} J.~Huang, M.~Leng and M.~Parlar, Demand functions in decision modeling: A comprehensive survey and research directions,  {\it Decis. Sci.,} {\bf 44} (2013), 557-609.
\bibitem{JRW2007variational} A.~Jofr{\'e}, R.T.~Rockafellar and R.J.B.~Wets, Variational inequalities and economic equilibrium,  {\it Math. Oper. Res.,} {\bf 32} (2007), 32-50.
\bibitem{K2009not} L.~Kilian, Not all oil price shocks are alike: Disentangling demand and supply shocks in the crude oil market,  {\it Am. Econ. Rev.,} {\bf 99} (2009), 1053-1069.
\bibitem{KM2014role} L.~Kilian and D.P.~Murphy, The role of inventories and speculative trading in the global market for crude oil,  {\it J. Appl. Econom.,} {\bf 29} (2014), 454-478.
\bibitem{MSS1982mathematical} F.H.~Murphy, H.D.~Sherali and A.L.~Soyster, A mathematical programming approach for determining oligopolistic market equilibrium,  {\it Math. Program.,} {\bf 24} (1982), 92-105.
\bibitem{PSS2017two} J.S.~Pang, S.~Sen and U.V.~Shanbhag, Two-stage non-cooperative games with risk-averse players,  {\it Math. Program.,} {\bf 165} (2017), 235-290.
\bibitem{PSL2015constructive} J.S.~Pang, C.L.~Su, Y.C.~Lee, A constructive approach to estimating pure characteristics demand models with pricing,  {\it Oper. Res.}, {\bf 63} (2015), 639-659.
\bibitem{PFW2016equilibrium} A.~Philpott, M.C.~Ferris and R.J.B.~Wets, Equilibrium, uncertainty and risk in hydro-thermal electricity systems,  {\it Math. Program.,} {\bf 157} (2016), 483-513.
\bibitem{OFCO2016cournot} J.~Outrata, M.C.~Ferris, M.~{\v{C}}ervinka and M.~Outrata, On Cournot-Nash-Walras equilibria and their computation,  {\it Set-Valued Var. Anal.}, {\bf 24} (2016), 387-402.
\bibitem{RPS1992the} C.W.~Richard, J.S.~Pang and S.E.~Richard,  {\it The linear complementarity problem}, Academic Press, 1992.
\bibitem{RS2018solving} R.T.~Rockafellar and J.~Sun, Solving monotone stochastic variational inequalities and complementarity problems by progressive hedging,  {\it Math. Program.},  {\bf174} (2019), 453-471.
\bibitem{RW1975stochastic} R.T.~Rockafellar and R.J.B.~Wets, Stochastic convex programming: Kuhn-Tucker conditions,  {\it J. Math. Econ.,} {\bf 2} (1975), 349-370
\bibitem{RW1976stochastic} R.T.~Rockafellar and R.J.B.~Wets, Stochastic convex programming: Basic duality,  {\it Pac. J. Math.,} {\bf 62} (1976), 173-195.
\bibitem{RW1976stochastic1} R.T.~Rockafellar and R.J.B.~Wets, Stochastic convex programming: Singular multipliers and extended duality singular multipliers and duality,  {\it Pac. J. Math.,} {\bf 62} (1976), 173-195.
\bibitem{RW1976stochastic2} R.T.~Rockafellar and R.J.B.~Wets, Stochastic convex programming: Relatively complete recourse and induced feasibility,  {\it SIAM. J. Control. Optim.,} {\bf 14} (1976), 574-589.
\bibitem{RW1991scenarios} R.T.~Rockafellar and R.J.B.~Wets, Scenarios and policy aggregation in optimization under uncertainty,  {\it Math. Oper. Res.,} {\bf 16} (1991), 119-147.
\bibitem{RW2009variational} R.T.~Rockafellar and R.J.B.~Wets,  {\it Variational analysis}, Springer Science \& Business Media, 2009.
\bibitem{RW2017stochastic} R.T.~Rockafellar and R.J.B.~Wets, Stochastic variational inequalities: single-stage to multistage,  {\it Math. Program.,} {\bf 165}, (2017), 331-360.
\bibitem{R1965existence} B.J. Rosen, Existence and uniqueness of equilibrium points for concave n-person games,  {\it Econometrica}, (1965), 520-534.
\bibitem{S1976exhaustible} S.W.~Salant, Exhaustible resources and industrial structure: A Nash-Cournot approach to the world oil market,  {\it J. Polit. Econ.,} {\bf 84} (1976), 1079-1093.
\bibitem{SDR2009lectures}  A.~Shapiro, D.~Dentcheva and A.~Ruszczy{\'n}ski,  {\it Lectures on stochastic programming: modeling and theory}, SIAM, 2009.
\bibitem{S2003monte} A.~Shapiro, Monte Carlo sampling methods,  {\it Handbooks in operations research and management science}, {\bf 10} (2003), 353-425.
\bibitem{S2009world} J.L.~Smith, World oil: market or mayhem?,  {\it J. Econ. Perspect.,} {\bf 23} (2009), 145-164.
\bibitem{SSC2017saa} H.~Sun, C.L.~Su and X.~Chen, SAA-regularized methods for multiproduct price optimization under the pure characteristics demand model,  {\it Math. Program.,} {\bf 165} (2017), 361-389.
\bibitem{XPOD2015complementarity} H.~Xu, J.S.~Pang, F.~Ord{\'o}nez and M.~Dessouky, Complementarity models for traffic equilibrium with ridesharing,  {\it Transport. Res. B-meth.,} {\bf 81} (2015), 161-182.
\bibitem{YAO2008modeling} J.~Yao, I.~Adler and S.S.~Oren, Modeling and computing two-settlement oligopolistic equilibrium in a congested electricity network,  {\it Oper. Res.,} {\bf 56} (2008), 34-47.
\bibitem{JP2015speculation} L.~Juvenal and I.~Petrella, Speculation in the oil market,  {\it J. Appl. Econom.,} {\bf 30} (2015), 621-649.
\end{thebibliography}
\end{document}